\documentclass[11pt]{amsart}
\usepackage{amssymb}
\usepackage{amsthm}

\usepackage{fullpage}
\usepackage{mathrsfs}
\usepackage{amsmath,amsfonts}
\usepackage{centernot}
\usepackage{mathtools}
\usepackage{enumerate}

\newcommand{\mc}{\mathcal}
\newcommand{\NN}{\mathbb N}
\newcommand{\ZZ}{\mathbb Z}
\newcommand{\RR}{\mathbb R}
\newcommand{\BB}{\mathbb B}
\newcommand{\I}[1][n]{[{#1}]}
\newcommand{\zero}{\mathbf{0}}
\newcommand{\one}{\mathbf{1}}
\newcommand{\ee}{\mathbf{e}}

\newcommand{\matid}{\mathbf{I}}
\newcommand{\jota}{\mathbf{J}}
\newcommand{\jt}[1][t]{\mathscr{J}_{#1}}
\newcommand{\cte}{\mathscr{C}}
\newcommand{\C}[2]{\cte_{#1}^{#2}}
\newcommand{\cnk}{\C{n}{k}}
\newcommand{\cn}{\C{n}{2}}
\newcommand{\qcn}{Q(\cn)}

\newcommand{\circuit}{\mathsf{C}}

\DeclarePairedDelimiter\card{\lvert}{\rvert} 
\newcommand{\Card}[1]{\card*{#1}}   

\DeclareMathOperator{\supp}{supp}
\newcommand{\car}[1][n]{\chi_{#1}}
\newcommand{\symdif}{\bigtriangleup}

\DeclareMathOperator{\STAB}{STAB}
\DeclareMathOperator{\conv}{conv}
\newcommand{\ch}{^\ast} 
\newcommand{\Q}{Q\ch}
\newcommand{\cov}{\mathcal{C}} 
\newcommand{\tra}{\mathcal{T}} 
\newcommand{\blk}{\mathbf{b}} 

\newcommand{\fast}{f^{\ast}}
\newcommand{\ExtPv}{v}
\newcommand{\ExtPu}{u}
\newcommand{\ExtPw}{w}
\newcommand{\SuppPv}{\supp \ExtPv}
\newcommand{\SuppPu}{\supp \ExtPu}
\newcommand{\SuppPw}{\supp \ExtPw}

\newcommand{\SuppVa}{\supp a}

\newcommand{\G}{\mathcal{G}} 
\newcommand{\ngh}{\leftrightarrow} 
\newcommand{\nngh}{\centernot\ngh} 

\newcommand{\arc}[2][n]{\left[#2\right]_{#1}}

\numberwithin{equation}{section}
\theoremstyle{plain}
   \newtheorem{thm}{Theorem}[section]
   \newtheorem{coro}[thm]{Corollary}
   \newtheorem{lem}[thm]{Lemma}
   \newtheorem{propo}[thm]{Proposition}
   
   \newtheorem{assums}[thm]{Assumptions}
   
\theoremstyle{definition}
   \newtheorem{defn}[thm]{Definition}
   \newtheorem{exam}[thm]{Example}
   \newtheorem{notat}[thm]{Notation}
   \newtheorem{rem}[thm]{Remark}
   
\newenvironment{enumcona}{%
   \begin{enumerate}[\quad(a)]}{\end{enumerate}}
\newcommand{\tref}[2]{#1~\ref{#2}}
\newcommand{\thref}[2]{#1}
\newcommand{\refp}[1]{(\ref{#1})} 
\newcommand{\trrefp}[3]{#1~\ref{#2}(\ref{#3})}

\usepackage[%
   colorlinks,
   linkcolor=blue,
   linktocpage,
   anchorcolor=blue,
   citecolor=blue,
   filecolor=blue,
   urlcolor=blue
   ]{hyperref}%
\renewcommand{\tref}[2]{
   \hyperref[#2]{#1~\ref*{#2}}}
\renewcommand{\thref}[2]{
   \hyperref[#2]{#1}}
\renewcommand{\refp}[1]{
   \hyperref[#1]{(\ref*{#1})}}
\renewcommand{\trrefp}[3]{
   \hyperref[#3]{#1~\ref*{#2}(\ref*{#3})}}

\providecommand{\arxiv}[2][]{\href{http://www.arxiv.org/abs/#2}{arXiv:#2}}

\title{Vertex adjacencies in the set covering polyhedron}

\author{N\'estor~E.~Aguilera}
\address{Facultad de Ingenier\'ia Qu\'imica (UNL).
   Santiago del Estero 2829,
   3000 Santa Fe, Argentina.}
\email{nestoreaguilera@gmail.com}

\author{Ricardo~D.~Katz}
\address{CONICET-CIFASIS.
   Bv.~27 de febrero 210 bis,
   2000 Rosario, Argentina.}
\email{katz@cifasis-conicet.gov.ar}

\author{Paola~B.~Tolomei}
\address{CONICET and
   Facultad de Ciencias Exactas, Ingenier\'ia y Agrimensura (UNR),
   Pellegrini 250, 2000 Rosario, Argentina.}
\email{ptolomei@fceia.unr.edu.ar}

\begin{document}

\begin{abstract}
We describe the adjacency of vertices of the (unbounded version of the) set covering polyhedron, in a similar way to the description given by Chv\'atal for the stable set polytope. We find a sufficient condition for adjacency, and characterize it with similar conditions in the case where the underlying matrix is row circular. We apply our findings to show a new infinite family of minimally nonideal matrices.
\end{abstract}


\catcode`\@=11
\@namedef{subjclassname@2010}{%
  \textup{2010} Mathematics Subject Classification}
\catcode`\@=12

\maketitle

\section{Introduction}
\label{sec:intro}

The stable set polytope of a graph $G$, $\STAB(G)$, is one of the most studied polyhedra related to set packing problems. In 1975, Chv\'atal~\cite{Ch75} gave a characterization of the adjacency of its vertices: the characteristic vectors of two stable sets $S$ and $S'$ of $G$ are adjacent in $\STAB(G)$ if, and only if, the subgraph of $G$ induced by $S\symdif S' =(S\setminus S')\cup (S'\setminus S)$ is connected. 
 
Since a characterization of vertex adjacency may provide more insight into the associated combinatorial problem, eventually becoming the basis for efficient algorithms, it is quite natural to try to extend Chv\'atal's result to other settings. This was done by several authors, on occasion in the context of the simplex method or related to the Hirsch conjecture, see for instance Hausmann and Korte~\cite{HK78}, Ikebe and Tamura~\cite{IT95}, Alfakih and Murty~\cite{AM98}, Matsui and Tamura~\cite{MT93}, or Michini and Sassano~\cite{MS13}. See also Michini~\cite{Mi12} and references therein.

Sometimes, it may be very difficult to test adjacency: Papadimitriou~\cite{Pa78} observed the difficulty of the adjacency problem for the traveling salesman polytope, later Chung~\cite{Ch80} obtained a similar result for the set covering polytope, whereas Matsui~\cite{Ma95} showed the NP-completeness of the non adjacency problem for the set covering polytope even though the matrix involved has exactly three ones per row. Thus, in contrast to the case of $\STAB(G)$, it is unlikely that a simple characterization of the adjacency of vertices of the set covering polyhedron may be given.

Nevertheless, in this work we go one step beyond the usual sufficient condition of connectivity of a certain graph, and give another condition which is also sufficient for the adjacency of vertices of the (unbounded version of the) set covering polyhedron, showing that more restrictive but similar conditions are also necessary in the case of row circular matrices. Thus, the adjacency problem for the set covering polyhedron is polynomial for these matrices.

This paper is organized as follows. After presenting the notation, some preliminary results and comments in \tref{Section}{sec:background}, in \tref{Section}{sec:graph} we present the graph which we associate with each pair of vertices of the set covering polyhedron defined by a binary matrix $A$. A sufficient condition for adjacency in terms of this graph is presented in \tref{Theorem}{thm:suf} of \tref{Section}{sec:suf}. In \tref{Section}{sec:nec} we establish a characterization of adjacency which applies to the case where $A$ is a row circular matrix (\tref{Theorem}{thm:CharactAdj}), and give an example (\tref{Example}{exam:ppfnd:13}) showing that our sufficient condition is not always necessary even for circulant matrices. Finally, in \tref{Section}{sec:mni} we apply our results to obtain a new infinite family of minimally nonideal matrices based on known minimally nonideal circulant matrices.

\section{Notation and preliminary results}
\label{sec:background}

Let us start by establishing some notation and definitions.

We denote by $\NN$ the set of natural numbers, $\NN = \{1,2,\dotsc\}$;
by $\ZZ$ the set of integers;
by $\RR$ the set of real numbers;
by $\BB$ the set of binary numbers, $\BB = \{0,1\}$;
and by $\I$ the set $\{1,2,\dots,n\}$.

The vectors in the canonical basis of $\RR^n$ will be denoted by $\ee_1,\dots,\ee_n$.
We denote by $\zero_n$ and $\one_n$ the vectors in $\RR^n$ with all zeroes and all ones, respectively, dropping the subindex $n$ if the dimension is clear from the context.

The scalar product in $\RR^n$ is denoted by a dot, so that, e.g., $x\cdot\ee_i = x_i$ for $x\in\RR^n$.
Given $x$ and $y$ in $\RR^n$, we say that $x$ \emph{dominates} $y$, and write $x\ge y$, if $x_i\ge y_i$ for all $i\in\I$.
$\card{X}$ denotes the cardinality of the finite set $X$.

The \emph{support} of $x\in\RR^n$ is the set
$\supp x = \{i\in\I \mid x_i \ne 0\}$.
Conversely,
given $X\subset\I$, its \emph{characteristic vector},
$\car(X)\in\BB^n$,
is defined by
\[
   \ee_i\cdot \car(X) = \begin{cases}
      1 & \text{if $i\in X$}, \\
      0 & \text{otherwise},
   \end{cases}
\]
so that $\supp (\car(X)) = X$.

Given a binary matrix
$A\in \BB^{m\times n}$,
the \emph{set covering polyhedron} associated with $A$, $\Q(A)$, is the convex hull of non-negative integer solutions of $Ax\ge\one$,
\begin{equation}
   \label{defn:Q:ast}
   \Q(A) = \conv \{x\in\ZZ^n\mid Ax\ge\one, x\ge\zero\},
\end{equation}
and we denote by $Q(A)$ the linear relaxation
\begin{equation}
   \label{defn:Q}
   Q(A) = \{x\in\RR^n \mid Ax\ge\one, x\ge\zero\}.
\end{equation}

When $A$ has precisely two ones per row, i.e., when $A$ is the edge-node incidence matrix of a graph $G$, the set covering polyhedron that we consider here is the unbounded version of the node cover polytope of $G$: 
\[
   \overline{\Q(A)} = \conv \{x\in\ZZ^n\mid Ax\ge\one,\one \ge x \ge \zero\} \; .
\]
We observe that in this case the function $x\mapsto \one -x$ affinely maps $\overline{\Q(A)}$ to 
\[
\STAB(G)= \conv \{x\in\ZZ^n\mid Ax\le\one,\one \ge x \ge \zero\} \; . 
\]
As mentioned in the introduction, for the latter polytope Chv\'atal proved the following characterization of vertex adjacency: 

\begin{thm}[\cite{Ch75}]
The characteristic vectors $\car(S)$ and $\car(S')$ of two stable sets of $G$ are adjacent in $\STAB(G)$ if, and only if, the subgraph of $G$ induced by $S\symdif S'$ is connected. 
\end{thm}

As a consequence of this and the fact that it can be proved that for any binary matrix $A$ two vertices of $\Q(A)$ are adjacent if and only if they are adjacent in $\overline{\Q(A)}$, see~\cite{AKT17}, we obtain: 

\begin{coro}\label{coro:chacarct:two:ones}
Suppose that every row of $A\in\BB^{m\times n}$ has exactly two ones, i.e., that $A$ is the edge-node incidence matrix of a graph $G$. Then, two distinct vertices $\ExtPv$ and $\ExtPv'$ of $\Q(A)$ (resp. of  $\overline{\Q(A)}$) are adjacent in $\Q(A)$ (resp. in $\overline{\Q(A)}$) if, and only if, the subgraph of $G$ induced by $\SuppPv\symdif \SuppPv'$ is connected. 
\end{coro}

In this paper we assume
that the binary matrix $A$
associated with the set covering polyhedra in~\eqref{defn:Q:ast} and~\eqref{defn:Q}
verifies the following assumptions:

\begin{assums}
\label{assums:A}
The matrix $A\in\BB^{m\times n}$ satisfies:
\begin{itemize}
\item
it has no dominating rows,

\item
it has between $2$ and $n-1$ ones per row,

\item
it has no column of all ones or of all zeroes.
\end{itemize}
\end{assums}

We denote by $C_t$ the support of the $t$-th row of $A$, and we let $\cov = \{C_1,\dots,C_m\}$.
\tref{Assumptions}{assums:A} imply
that $2\le\card{C}\le n-1$ for every $C\in\cov$,
that $\bigcup_{C\in\cov} C = \I$,
and that $\cov$ is a \emph{clutter} in the nomenclature of \cite{CN94}.

The vertices of $\Q(A)$
are the binary vertices of $Q(A)$.
They form the {\em blocker} of $A$, $\blk(A)$,
and their supports are the {\em minimal transversals} of $\cov$, which we denote by $\tra$.
That is, $T\in\tra$ if and only if $T\cap C\ne\emptyset$ for all $C\in\cov$ and if $R\subset T$ with $R\cap C\ne\emptyset$ for all $C\in\cov$ then $R = T$.
Notice that
\begin{itemize}
\item
$C\cap T \ne\emptyset$ for all $C\in\cov$ and $T\in\tra$.

\item
For every $T\in\tra$ and $p\in T$,
there exists $C\in\cov$ such that $C\cap T = \{p\}$.
\end{itemize}
As $\blk(\blk(A)) = A$ if $A$ is a clutter matrix,
we also have:
\begin{itemize}
\item
For every $C\in\cov$ and $p\in C$
there exists $T\in\tra$ such that $C\cap T = \{p\}$.
\end{itemize}

\begin{rem}
\label{rem:card:T}
Notice that \tref{Assumptions}{assums:A} also imply that $\tra$ has properties similar to those of $\cov$: $2\le\card{T}\le n-1$ for every $T\in\tra$, and $\bigcup_{T\in\tra} T = \I$.
\end{rem}

A \emph{convex combination} of the points $x^1, \ldots ,x^\ell$ of $\RR^n$ is a point of the form $\sum_{k\in \I[\ell]} \lambda_k x^k$,
where $\sum_{k\in \I[\ell]} \lambda_k = 1$ and
$\lambda_k\ge 0$ for $k\in \I[\ell]$.
The combination is \emph{strict} if
all $x^k$ are different and
$0 < \lambda_k < 1$ for all $k\in \I[\ell]$.

The following result is well known and we will use it to prove adjacency:

\begin{propo}\label{propo:adys:1}
Suppose
$P = \{x\in\RR^n \mid Ax\ge b, x\ge\zero\}$,
where
$A\in\RR^{m\times n}$
has non-negative entries
and $b\ge\zero$.
If $\ExtPv$ and $\ExtPv'$ are distinct vertices of $P$, the following are equivalent:
\begin{enumcona}
\item\label{propo:adys:1:a}
$\ExtPv$ and $\ExtPv'$ are adjacent in $P$.

\item\label{propo:adys:1:b}
If a strict convex combination $\sum_{k\in \I[\ell]} \lambda_k x^k$ of points of $P$ belongs to the segment with endpoints $\ExtPv$ and $\ExtPv'$,
then $x^k$ also belongs to this segment for all $k\in \I[\ell]$.

\item\label{propo:adys:1:c}
If $y = \sum_{k\in \I[\ell]} \lambda_k \ExtPu^k$
is a strict convex combination of vertices $\ExtPu^1, \ldots, \ExtPu^\ell$ of $P$ and
$y\le \frac{1}{2}\,(\ExtPv + \ExtPv')$,
then $\ell =2$ and, without loss of generality, $\ExtPu^1=\ExtPv$ and $\ExtPu^2=\ExtPv'$.
\end{enumcona}
\end{propo}

We point out that it is possible to relax the usual condition $y=\frac{1}{2}\,(\ExtPv + \ExtPv')$ to the inequality $y\le \frac{1}{2}\,(\ExtPv + \ExtPv')$ in \trrefp{Proposition}{propo:adys:1}{propo:adys:1:c} due to the fact that we assume that $A$ has non-negative entries, and so the polyhedron $P$ satisfies the following property: $x\in P$ and $z\geq x$ imply $z\in P$. 

To prove that the vertices $v$ and $v'$ of $\Q(A)$ are not adjacent, it will be convenient to make use of a variant of \tref{Proposition}{propo:adys:1}.
Namely, suppose $v$ and $v'$ can be decomposed as nontrivial sums of binary vectors:
\[
v = z + c + d
\quad \text{and} \quad
v' = z + c' + d',
\]
so that $z$ is a ``common part'',
and the remaining parts are split into two: $c$ and $d$ for $v$ and $c'$ and $d'$ for $v'$.
Suppose also that by interchanging $d$ and $d'$ we obtain two points of $\Q(A)$:
\[
x = z + c + d' = v - d + d'
\quad \text{and} \quad
x' = z + c' + d = v' - d' + d .
\]
Since $\frac{1}{2}\,(x + x') = \frac{1}{2}\,(\ExtPv + \ExtPv')$, if we could assure that either $x$ or $x'$ does not belong to the segment with endpoints $v$ and $v'$, then by   the equivalence of~\refp{propo:adys:1:a} and~\refp{propo:adys:1:b} of \tref{Proposition}{propo:adys:1} when $P=\Q(A)$, we would conclude that $v$ and $v'$ are not adjacent in $\Q(A)$.
This is the idea behind the next lemma.

\begin{propo}\label{propo:adys:2}
Let $\ExtPv$ and $\ExtPv'$ be distinct vertices of $\Q(A)$. Suppose there exist $d$ and $d'$ in $\RR^n$ such that:
\begin{itemize}
\item
$\zero\leq d \leq \ExtPv$ and $\zero\lneqq d'$,

\item
$\ExtPv\cdot d' = 0$,

\item
$x= \ExtPv - d + d'$ and $x'= \ExtPv' - d' + d$ are elements of $\Q(A)$,

\item
$x$ is binary and different from $\ExtPv'$.

\end{itemize}

Then $\ExtPv$ and $\ExtPv'$ are not adjacent in $\Q(A)$.
\end{propo}

\begin{proof}
We notice first that
\( 0\le (\ExtPv - d)\cdot d' \le \ExtPv\cdot d' = 0, \)
which implies \((\ExtPv - d)\cdot d' = 0\).

If $x$ was equal to $\ExtPv$, we would have
\[
   \begin{aligned}
   0
      &= \ExtPv\cdot d'
         && \text{by hypothesis,} \\
      &= x\cdot d'
         && \text{since we are assuming $x = \ExtPv$,} \\
      &= (\ExtPv - d + d')\cdot d'
         \quad
         && \text{by definition of $x$,} \\
      &= d'\cdot d'
         && \text{since $(\ExtPv - d)\cdot d' = 0$,} \\
      &> 0
         && \text{since $d'\ne\zero$,}
   \end{aligned}
\]
i.e., we obtain a contradiction. Thus, $x$ is different from $\ExtPv$.

Given that a segment with endpoints in $\BB^n$ cannot contain other binary points, and that $x$ is binary and different from  $\ExtPv$ and $\ExtPv'$, it follows that $x$ cannot belong to the segment with endpoints $\ExtPv$ and $\ExtPv'$. The result now follows from the equivalence of~\refp{propo:adys:1:a} and~\refp{propo:adys:1:b} of \tref{Proposition}{propo:adys:1} when $P=\Q(A)$, since $\frac{1}{2}\,(x + x') = \frac{1}{2}\,(\ExtPv + \ExtPv')$. 
\end{proof}

\section{The joint saturation graph}
\label{sec:graph}

The following definitions are essential in this paper.

\begin{defn}\label{defn:G}
Given a matrix $A\in\BB^{m\times n}$ and the associated clutter $\cov$ as described in the previous section,
let $\ExtPv$ and $\ExtPv'$ be distinct vertices of $\Q(A)$. We construct a simple undirected graph $\G_A(\ExtPv,\ExtPv')$ depending on $\ExtPv$, $\ExtPv'$ and $A$, called the {\em joint saturation graph of $\ExtPv$ and $\ExtPv'$ (with respect to $A$)}, by the following setup:
\begin{itemize}
\item
the set of nodes of $\G_A(\ExtPv,\ExtPv')$ is
\[
\SuppPv\symdif \SuppPv' = (\SuppPv \setminus \SuppPv') \cup (\SuppPv'\setminus \SuppPv) \; ,
\]

\item
$\G_A(\ExtPv,\ExtPv')$ is bipartite with partite sets
\[
\SuppPv \setminus \SuppPv' \quad\text{and}\quad \SuppPv'\setminus \SuppPv\; ,
\]
\item
$p\in \SuppPv \setminus \SuppPv'$ and $p'\in \SuppPv'\setminus \SuppPv$ are neighbors in $\G_A(\ExtPv,\ExtPv')$ if there exists $C_t\in \cov$ such that
\begin{equation}
\label{equ:edge}
   C_t\cap \SuppPv = \{p\} \quad\text{and}\quad C_t\cap \SuppPv' = \{p'\}.
\end{equation}
\end{itemize}
\end{defn}
 
See Figure~\ref{figure2} below for an illustration of~\tref{Definition}{defn:G}.  

Following West~\cite{We01},
we will denote by $p\ngh p'$ and $p\nngh p'$ whether $p$ and $p'$ are neighbors in $\G_A(\ExtPv,\ExtPv')$ or not, respectively. A path of $\G_A(\ExtPv,\ExtPv')$ will be called {\em even} (resp. {\em odd}) if it contains an even (resp. odd) number of edges.

The name joint saturation graph comes from the fact that each edge of $\G_A(\ExtPv,\ExtPv')$ corresponds to an inequality in $A x \geq \one$ which is saturated, i.e., satisfied with equality, by both vertices $\ExtPv$ and $\ExtPv'$. Observe that there may exist inequalities in $A x \geq \one$ which are saturated by both vertices $\ExtPv$ and $\ExtPv'$ and do not correspond to edges of $\G_A(\ExtPv,\ExtPv')$, because the saturation can be due to a coordinate in $\SuppPv \cap \SuppPv'$. 

In what follows, when the matrix $A$ is clear from the context, we will simply write $\G(\ExtPv,\ExtPv')$. 

\begin{rem}
\label{rem:contraction} 
Let us recall that given a nontrivial proper subset $I$ of $\I$, the \emph{contraction minor} $A/I$ is obtained by eliminating the columns of $A$ with indices in $I$, and then removing any dominating row that might appear.

Note that $\ExtPv$ and $\ExtPv'$ are adjacent in $\Q(A)$ if, and only if, they are adjacent in the face $\{x\in \Q(A)\mid  x_i =0 \text{ for } i\in I\}$, where $I=\I \setminus (\SuppPv \cup \SuppPv')$. Observe that the projection of this face on the coordinates in $\SuppPv \cup \SuppPv'$ is given by $\Q(A/I)$, and that $\G_A(\ExtPv,\ExtPv')$ coincides with $\G_{A/I}(\ExtPv_{\I \setminus I},\ExtPv'_{\I \setminus I})$, where $\ExtPv_{\I \setminus I}$ and $\ExtPv'_{\I \setminus I}$ denote the projections of $\ExtPv$ and $\ExtPv'$ on the coordinates in $\SuppPv \cup \SuppPv'$, respectively. 

Now, let $\bar{A}$ be the submatrix of $A/I$ consisting of the rows with precisely two ones, and let $G$ be the graph whose edge-node incidence matrix is $\bar{A}$. It can be proved that $\G_A(\ExtPv,\ExtPv')$ coincides with the subgraph of $G$ induced by the nodes in $\SuppPv\symdif \SuppPv'$.  
\end{rem}

\begin{rem}\label{rem:graph:1}
If $\ExtPv$ and $\ExtPv'$ are distinct vertices of $\Q(A)$, then:
\begin{itemize}
\item
$\SuppPv \setminus \SuppPv'\ne\emptyset$ since we cannot have $\ExtPv\le \ExtPv'$.
Similarly, we have $\SuppPv' \setminus \SuppPv\ne\emptyset$.

\item
Consequently, $\G(v,v')$ has at least two nodes.

\item
If $\G(v,v')$ has exactly two nodes, then it is connected: if $p\in\SuppPv \setminus \SuppPv'$ and $p'\in\SuppPv' \setminus \SuppPv$ are the two nodes, $R = \supp v\cap \supp v'$, and $C_t\in\cov$ is such that $C_t\cap\SuppPv = \{p\}$, then necessarily $C_t\cap R = \emptyset$ and $C_t\cap \SuppPv' = \{p'\}$.
\end{itemize}
\end{rem}

Sufficient and necessary conditions for the adjacency of vertices of $\Q(A)$ will be given in terms of properties of the joint saturation graph. With this aim, it is convenient to make the following definition:

\begin{defn}\label{defn:partly}
A bipartite graph is said to be \emph{partite-connected} if one of its partite sets is contained in a component, and it is said to be \emph{almost-connected} if it has exactly two components, one of which is an isolated node.
\end{defn}

Observe that according to our definition, a connected bipartite graph is partite-connected but not almost-connected, and an almost-connected graph is always partite-connected.

\section{A sufficient condition for adjacency}
\label{sec:suf}

In this section we present a sufficient condition for the adjacency of two vertices of $\Q(A)$ in terms of their joint saturation graph. For this, we will need the following lemma.

\begin{lem}
\label{lem:graph:1}
Let $\ExtPv$ and $\ExtPv'$ be distinct vertices of $\Q(A)$, and let $\G(\ExtPv,\ExtPv')$ be their joint saturation graph. Suppose $y = \sum_{k\in \I[\ell]} \lambda_k \ExtPu^k$
is a strict convex combination of vertices
$\ExtPu^1,\ldots ,\ExtPu^\ell$ of $\Q(A)$
such that $y\le \frac{1}{2}\,(\ExtPv + \ExtPv')$.
Then, for each $k\in \I[\ell]$, we have:
\begin{enumcona}
\item\label{lem:graph:1:a}
$\SuppPu^k \subseteq \SuppPv \cup \SuppPv'$.

\item\label{lem:graph:1:b}
$\Card{\{p,p'\}\cap \supp \ExtPu^k} = 1$
whenever $p$ and $p'$ are neighbors in $\G(\ExtPv,\ExtPv')$.
\end{enumcona}
\end{lem}

\begin{proof}
Let us write $z = \frac{1}{2}\,(\ExtPv + \ExtPv')$.

If $q\in \supp \ExtPu^k$ we must have $\ExtPu^k\cdot \ee_q > 0$, and therefore
$y\cdot\ee_q > 0$ since the convex combination for $y$ is strict.
Hence $z\cdot\ee_q > 0$ as $z\ge y$.
Therefore, $q\in \SuppPv \cup \SuppPv'$.

For the second part, let $a$ be a row of $A$ such that~\eqref{equ:edge} holds with $C_t = \supp a$.
Since $\SuppPu^k \subseteq \SuppPv \cup \SuppPv'$, by~\eqref{equ:edge} it follows that $p$ and $p'$ are the only elements of $C_t$ which can belong to $\SuppPu^k$.
Then, from $\{p, p'\}\subset C_t$ we conclude that
\[
\{p, p'\} \cap \SuppPu^k =C_t\cap \SuppPu^k .
\]
Moreover, by~\eqref{equ:edge} we have $a\cdot z = 1$, and since $z\ge y$, it follows that $a\cdot y \le 1$.
On the other hand, as $y\in\Q(A)$, we have $a\cdot y\ge 1$, and therefore $a\cdot y = 1$.
Similarly, since $y$ is a strict convex combination of the points
$\ExtPu^1,\ldots ,\ExtPu^\ell$
and $a\cdot \ExtPu^h\ge 1$ for all $h\in \I[\ell]$, we conclude that $
a\cdot \ExtPu^h = 1$ for all $h\in \I[\ell]$. Thus, in particular we have $\Card{C_t\cap \SuppPu^k} = 1$, proving the lemma.\qedhere
\end{proof}

\begin{lem}
\label{lem:CondImplicaAdy}
Let $\ExtPv$ and $\ExtPv'$ be distinct vertices of $\Q(A)$.
Suppose $y = \sum_{k\in \I[\ell]} \lambda_k \ExtPu^k$
is a strict convex combination of vertices
$\ExtPu^1,\ldots ,\ExtPu^\ell$ of $\Q(A)$
such that $y\le \frac{1}{2}\,(\ExtPv + \ExtPv')$. 
If the joint saturation graph of $\ExtPv$ and $\ExtPv'$ is partite-connected, then $\ell =2$ and, without loss of generality, $\ExtPu^1=\ExtPv$ and $\ExtPu^2=\ExtPv'$.
\end{lem}

\begin{proof}
To prove the lemma it is enough to show that, for each $k\in \I[\ell]$, either $\ExtPu^k = \ExtPv$ or $\ExtPu^k = \ExtPv'$.

Without loss of generality, assume that the partite set $\SuppPv \setminus \SuppPv'$ is contained in a component of $\G(\ExtPv,\ExtPv')$.

Let $I=\left\{ k\in \I[\ell] \mid \SuppPu^k \cap (\SuppPv \setminus \SuppPv')  = \emptyset\right\}$ and $J= \I[\ell] \setminus I$.

Observe that for $k\in I$ we must have $\SuppPu^k \subseteq \SuppPv'$ as
$\SuppPu^k \subseteq \SuppPv\cup \SuppPv'$ by \trrefp{Lemma}{lem:graph:1}{lem:graph:1:a}. Thus, since $\ExtPv'$ and $\ExtPu^k$ are binary vertices of $\Q(A)$, we conclude that
\begin{equation}
   \label{equ:nestor:1}
   \ExtPu^k = \ExtPv'
   \quad\text{for all }k\in I.
\end{equation}

On the other hand, for $k\in J$, let us fix $p\in \SuppPu^k \cap (\SuppPv \setminus \SuppPv')$.
Since $\SuppPv \setminus \SuppPv'$ is contained in a component of $\G(\ExtPv,\ExtPv')$,
for each $q\in (\SuppPv \setminus \SuppPv')\setminus \{p\}$ there exists a path
$p = p_1$, $p'_2$,$\dots$, $p'_h$, $p_h = q$ connecting $p$ and $q$,
where $p'_i \in \SuppPv' \setminus \SuppPv$ and $p_i\in \SuppPv \setminus \SuppPv'$ for $i = 2,\dots,h$.
Using repeatedly \trrefp{Lemma}{lem:graph:1}{lem:graph:1:b}, we see that $p_1 = p\in \SuppPu^k$, $p'_2\notin \SuppPu^k$, $p_2\in \SuppPu^k$, and so on, i.e., $p_i\in \SuppPu^k$ and $p'_i\notin \SuppPu^k$ for all $i = 2,\dots, h$. Thus, in particular we have $q= p_h \in \SuppPu^k$. Since this holds for any $q\in (\SuppPv \setminus \SuppPv')\setminus \{p\}$, we conclude that
\begin{equation}
   \label{equ:nestor:2}
   \SuppPv \setminus \SuppPv'\subseteq \SuppPu^k
   \quad\text{for all } k\in J.
\end{equation}

Consider now any $q\in \SuppPv \setminus \SuppPv'$. From~\eqref{equ:nestor:1}, \eqref{equ:nestor:2} and the fact that $\frac{1}{2}\,(\ExtPv + \ExtPv')\geq y= \sum_{k\in \I[\ell]} \lambda_k \ExtPu^k$, we obtain
\[
\frac{1}{2}=\frac{1}{2}\,(\ExtPv_{q} + \ExtPv'_{q})\geq \sum_{k\in I} \lambda_k \ExtPu^k_{q} + \sum_{k\in J} \lambda_k \ExtPu^k_{q}=\sum_{k\in J} \lambda_k\; ,
\]
and so $\sum_{k\in I} \lambda_k=1-\sum_{k\in J} \lambda_k\geq \frac{1}{2}$. Since using~\eqref{equ:nestor:1} we also have
\[
\frac{1}{2}\,(\ExtPv + \ExtPv')\geq \sum_{k\in I} \lambda_k \ExtPu^k + \sum_{k\in J} \lambda_k \ExtPu^k=\sum_{k\in I} \lambda_k \ExtPv' + \sum_{k\in J} \lambda_k \ExtPu^k,
\]
we conclude that $\frac{1}{2}\, \ExtPv \geq \sum_{k\in J} \lambda_k \ExtPu^k$, which implies $\SuppPu^k\subseteq \SuppPv$ for all $k\in J$. Therefore, since $\ExtPv$ and $\ExtPu^k$ are binary vertices of $\Q(A)$, it follows that $\ExtPu^k = \ExtPv$ for all $k\in J$. This completes the proof. \qedhere
\end{proof}

Using \tref{Proposition}{propo:adys:1} and \tref{Lemma}{lem:CondImplicaAdy}, we obtain the main result of this section:

\begin{thm}
\label{thm:suf}
If the joint saturation graph of two distinct vertices of $\Q(A)$ is partite-connected, then these vertices are adjacent in $\Q(A)$.
\end{thm}
 
\tref{Theorem}{thm:suf} implies that if $\G(\ExtPv,\ExtPv')$ is connected, then $\ExtPv$ and $\ExtPv'$ are adjacent in $\Q(A)$. This
sufficient condition can be alternatively derived from \tref{Corollary}{coro:chacarct:two:ones} and \tref{Remark}{rem:contraction}. To see this, let $I$, $\bar{A}$, $G$, $\ExtPv_{\I \setminus I}$ and $\ExtPv'_{\I \setminus I}$ be defined as in~\tref{Remark}{rem:contraction}, and consider the subgraph of $G$ induced by $\SuppPv\symdif \SuppPv'$, which coincides with $\G(\ExtPv,\ExtPv')$. If this subgraph is connected, then $\ExtPv_{\I \setminus I}$ and $\ExtPv'_{\I \setminus I}$ are adjacent in $\Q(\bar{A})$ (by \tref{Corollary}{coro:chacarct:two:ones}), which implies they are also adjacent in $\Q(A/I)$, and so $\ExtPv$ and $\ExtPv'$ are adjacent in $\Q(A)$ (by \tref{Remark}{rem:contraction} we know that $\ExtPv_{\I \setminus I}$ and $\ExtPv'_{\I \setminus I}$ are adjacent in $\Q(A/I)$ if and only if $\ExtPv$ and $\ExtPv'$ are adjacent in $\Q(A)$). 

\section{Characterization of vertex adjacency for row circular matrices}
\label{sec:nec}

As mentioned in the \thref{introduction}{sec:intro}, the sufficient condition of \tref{Theorem}{thm:suf} is not always necessary.
In this section we show that the converse of that theorem is true
when the matrix $A$ is row circular.
Actually, in this case we will give a much more detailed characterization
in terms of properties of the joint saturation graph.
We will also show that being partite-connected is far from being a necessary condition when we consider the similar class of circulant matrices.

Let us first recall that the \emph{circulant matrix} $\cte(c)\in \RR^{n\times n}$ associated with a vector $c = (c_1,\dots,c_n)\in\RR^n$ is defined as
\[
   \cte(c) = \cte(c_1,\dots,c_n) =
      \begin{bmatrix}
      c_1 & c_2 & \dots & c_n \\
      c_n & c_1 & \dots & c_{n-1} \\
      \vdots & \vdots & \ddots & \vdots \\
      c_2 & c_3 & \dots & c_1
      \end{bmatrix},
\]
where each row is a right rotation (shift) of the previous one.

Following Bartholdi et al.~\cite{BOR80}, we will say that a binary vector is \emph{circular} if its ones occur consecutively, where the first entry and the last entry of the vector are considered to be consecutive, or, alternatively, if either the ones are all consecutive or the zeroes are all consecutive.
A binary matrix is said to be \emph{row circular} if all its rows are circular.

To deal with circular vectors of $\BB^n$, it is convenient to consider circular arcs of $\I$: for $i,j \in \I$, the \emph{(directed) circular arc} $\arc{i,j}$ is defined as
\[
\arc{i,j} = 
\begin{cases}
\{i,\dots,j\} & \text{if } i \le j , \\
\{i,\dots,n\} \cup \{1,\dots,j\} & \text{if } j < i .
\end{cases}
\]
Thus, the rows of a row circular matrix $A\in\BB^{m\times n}$ may be considered as the characteristic vectors of circular arcs of $\I$.

\begin{exam}
\label{exam:cnk:1}
A particularly interesting case of row circular matrices is that of the \emph{consecutive ones circulant matrices} $\cnk$,
where the sets in the clutter $\cov$ are of the form
\[
C_t = \{t, t + 1,\dots, t + k - 1\}, \quad t\in\I,
\]
(sums are taken modulo $n$ with values in $\I$) so that $\cnk$ is also circulant.
For example,
\[
   \C{3}{2} = \begin{bmatrix}
      1 & 1 & 0 \\
      0 & 1 & 1  \\
      1 & 0 & 1
      \end{bmatrix} = \cte(1,1,0).
\]
\end{exam}

In the remainder of this section, we will make the following assumptions:

\begin{assums}
\label{assums:B}
The matrix $A$ is row circular and satisfies \tref{Assumptions}{assums:A}, and $\cov$ is the associated clutter (as described in \tref{Section}{sec:background}).
\end{assums}

Thus, in particular we have $m = \card{\cov}\ge 2$,
$2\le\card{C_t}\le n-1$ for all $t\in \I[m]$, and
$2\le \ExtPv\cdot \one \le n-1$ for every vertex $v$ of $\Q(A)$.

We will also use the following convention:

\begin{notat}
For $\ExtPv\in\BB^n$ we will write $\SuppPv = \{p_1,p_2,\dots,p_r\}$,
with $ p_1 < p_2 <\dots< p_r$;
and for $h\notin\I[r]$ we let $p_h = p_i$
with $i\in\I[r]$ and $h\equiv i \pmod{r}$.
Similarly, for $\ExtPv'\in\BB^n$ we will write
$\SuppPv' =  \{p'_1,\dots,p'_{r'}\}$
with $p'_1 < \dots < p'_{r'}$, etc.

We will also consider that operations involving elements of the support of a vector of $\BB^n$, such as $p_i+1$, are taken modulo $n$ with values in $\I$.
\end{notat}

\begin{rem}\label{rem:graph:2}
When $A$ is row circular, if $p\ngh p'$ (i.e., \eqref{equ:edge} is satisfied), exactly one of the circular arcs $\arc{p,p'}$ or $\arc{p',p}$ is such that its intersection with $\SuppPv\cup \SuppPv'$ is $\{p,p'\}$. Indeed, by~\tref{Remark}{rem:card:T} we know that $\SuppPv$ has at least two elements, and so there exists $q\in \SuppPv$ such that  $q\neq p$ (note that we also have $q\neq p'$ because $p'\in \SuppPv' \setminus \SuppPv$ by~\tref{Definition}{defn:G}). Besides, since $A$ is row circular, the circular arc $C_t$ satisfying~\eqref{equ:edge} contains either $\arc{p,p'}$ or $\arc{p',p}$. Then, if $q\in  \arc{p,p'}$, we have $\{ q ,p , p'\}\subset (\SuppPv\cup \SuppPv')\cap \arc{p,p'}$, and so by~\eqref{equ:edge} we conclude that $\arc{p',p} \subset C_t$ and $\{p,p'\} = (\SuppPv\cup \SuppPv')\cap \arc{p',p}$. Otherwise, i.e., if $q\in  \arc{p',p}$, we have $\{ q ,p , p'\}\subset (\SuppPv\cup \SuppPv')\cap \arc{p',p}$, and so from~\eqref{equ:edge} it follows that $\arc{p,p'} \subset C_t$ and $\{p,p'\} = (\SuppPv\cup \SuppPv')\cap \arc{p,p'}$.
\end{rem}

Let us state now some simple results.

\begin{lem}
\label{lem:nec:1}
Suppose $C_t\in\cov$ and $\ExtPv$ is a vertex of $\Q(A)$.
Then
\[
1\le \card{C_t\cap \SuppPv }\le 2.
\]
Moreover, if $\card{C_t\cap \SuppPv } = 2$,
then $C_t\cap \SuppPv = \{p_i, p_{i+1}\}$ for some $i$.
\end{lem}

\begin{proof}
We obviously have $\card{C_t\cap \SuppPv }\ge 1$ because $\SuppPv$ is a transversal. 
If we had $\card{C_t\cap \SuppPv }\ge 3$, since $C_t$ is a circular arc, 
there would exist three different elements $p_i$, $p_j$ and $p_h$ of $\SuppPv$ such that $p_j\in \arc{p_i,p_h}\subset C_t$. Besides, as $\SuppPv$ is a minimal transversal,
there exists $C_s\in\cov$ such that $C_s\cap \SuppPv = \{p_j\}$.
Then, since $C_s$ is a circular arc which contains $p_j$ but does not contain $p_i$ nor $p_h$,
we would have $C_s\subsetneq \arc{p_i,p_h}\subset C_t$, contradicting the fact that $A$ has no dominating rows (by \tref{Assumptions}{assums:A}). Thus, we necessarily have $\card{C_t\cap \SuppPv }\le 2$. 

The last part follows from the fact that $C_t$ is a circular arc.\qedhere
\end{proof}

\begin{lem}\label{lem:nec:4}
Let $\ExtPv$ and $\ExtPv'$ be distinct vertices of $\Q(A)$. Suppose $p_i\in \SuppPv \setminus \SuppPv'$ is an isolated node of the joint saturation graph $\G(\ExtPv,\ExtPv')$. If $C_t\in\cov$ and $C_t\cap \SuppPv = \{p_i \}$, then $\card{C_t\cap (\SuppPv' \setminus \SuppPv)} = 2$ and $C_t\cap \SuppPv' = C_t\cap (\SuppPv' \setminus \SuppPv)$. 
\end{lem}

\begin{proof}
In the first place, observe that $C_t\cap \SuppPv' = C_t\cap(\SuppPv' \setminus \SuppPv)$ because $C_t\cap \SuppPv = \{p_i\}$ and $p_i\in \SuppPv \setminus \SuppPv'$. Then, by \tref{Lemma}{lem:nec:1}, we have  
$1\le \card{C_t\cap (\SuppPv' \setminus \SuppPv)}\le 2$. Now, note that we cannot have 
$\card{C_t\cap (\SuppPv' \setminus \SuppPv)}=1$, because by~\tref{Definition}{defn:G} (see in particular~\eqref{equ:edge}) that would mean that there exists an edge in  $\G(\ExtPv,\ExtPv')$ connecting $p_i$ with the unique element of $C_t\cap (\SuppPv' \setminus \SuppPv)$, contradicting the fact that $p_i$ is an isolated node of $\G(\ExtPv,\ExtPv')$. \qedhere     
\end{proof}

The next lemmas provide simple properties of the joint saturation graph when $A$ is row circular, which we will need to establish the characterization of vertex adjacency for $\Q(A)$.

\begin{lem}
\label{lem:nec:5}
The joint saturation graph $\G(\ExtPv,\ExtPv')$ of two distinct vertices $\ExtPv$ and $\ExtPv'$ of $\Q(A)$ has the following properties:
\begin{enumcona}

\item\label{lem:nec:5:2}
If $p_i\ngh p'_j$ and $p_h \ngh p'_j$ in $\G(\ExtPv,\ExtPv')$, with $h\neq i$, then we must have
either $h = i - 1$ or $h = i + 1$.

\item\label{lem:nec:5:3}
If $p_i\ngh p'_j$, $p_{i+1}\ngh p'_j$ and $\card{\SuppPv} > 2$, then $p'_j \in \arc{p_i, p_{i+1}}$ and there are no other elements of $\SuppPv'$ (or $\SuppPv$) in $\arc{p_i, p_{i+1}}$.\footnote{Notice that if $\card{\SuppPv} = 2$, then $p'_j$ could be in either $\arc{p_1, p_2}$ or $\arc{p_2, p_1}$.}

\item\label{lem:nec:5:1} 
The nodes of $\G(\ExtPv,\ExtPv')$ have degree at most $2$. 

\item\label{lem:nec:5:4} 
Each component of $\G(\ExtPv,\ExtPv')$ must be either a cycle or a path (including isolated nodes). 

\item\label{lem:nec:5:5}
If a component of $\G(\ExtPv,\ExtPv')$ is a cycle,
then its set of nodes is equal to $\SuppPv\cup \SuppPv'$, and we have $\SuppPv\cap \SuppPv' = \emptyset$.
In particular, $\G(\ExtPv,\ExtPv')$ is connected. 

\end{enumcona}
\end{lem}

\begin{proof}
Let us assume that $C_s\in \cov$ is such that
\begin{equation}\label{Arco1}
   C_s\cap \SuppPv = \{p_i\} \; \makebox{ and }\;
   C_s\cap \SuppPv' = \{p'_j\}\; 
\end{equation}
and $C_r\in \cov$ such that
\begin{equation}\label{Arco2}
   C_r\cap \SuppPv = \{p_h\} \; \makebox{ and } \;
   C_r\cap \SuppPv' = \{p'_j\}.
\end{equation}
Since $p'_j\in C_s\cap C_r $, it follows that $C_s\cup C_r $ is a circular arc. By~\eqref{Arco1} and~\eqref{Arco2}, this circular arc intersects $\SuppPv$ only at $p_i$ and $p_h$, and hence these are consecutive elements of $\SuppPv$.
This shows~\refp{lem:nec:5:2}.

To prove~\refp{lem:nec:5:3}, let $C_s$ and $C_r$ be circular arcs satisfying~\eqref{Arco1} and~\eqref{Arco2}, with $h=i+1$ in~\eqref{Arco2}. Then, as above, we can conclude that $C_s\cup C_r $ is a circular arc which intersects $\SuppPv$ only at $p_i$ and $p_{i+1}$, and $\SuppPv'$ only at $p'_j$. Besides, since $\card{\SuppPv} > 2$, there exists an element $p_k$ in $\arc{p_{i+1}, p_i}\cap \SuppPv$ which is different from $p_i$ and $p_{i+1}$. Observe that \tref{Remark}{rem:graph:2} would not hold for $p_i\ngh p'_j$ if $p'_j$ belonged to $\arc{p_{i+1}, p_k}$. Analogously, note that \tref{Remark}{rem:graph:2} would not hold for $p_{i+1}\ngh p'_j$ if $p'_j$ belonged to $\arc{p_k, p_i}$. Therefore, we conclude that $p'_j\in \arc{p_i, p_{i+1}}$. Finally, since $p_k \in \arc{p_{i+1}, p_i}$ and $C_s\cup C_r $ is a circular arc which contains $p_i$ and $p_{i+1}$ but does not contain $p_k$, we have $\arc{p_i, p_{i+1}} \subset C_s\cup C_r $. Thus, from~\eqref{Arco1} and~\eqref{Arco2} (recall that $h=i+1$ in~\eqref{Arco2}), it follows that $p_i$, $p_{i+1}$ and $p'_j$ are the only elements of $\SuppPv \cup \SuppPv'$ in $\arc{p_i , p_{i+1}}$. 

Assume that a node of $\G(\ExtPv,\ExtPv')$, for instance $p'_j$, has degree strictly greater than $2$. Let $p_i$, $p_h$ and $p_k$ be three different elements of $\SuppPv \setminus \SuppPv'$ such that $p_i\ngh p'_j$, $p_h\ngh p'_j$ and $p_k\ngh p'_j$. By~\refp{lem:nec:5:2} we can assume, without loss of generality, that $h=i-1$ and $k=i+1$.  Then, by~\refp{lem:nec:5:3} we have $p'_j\in \arc{p_{i-1}, p_i}$ and $p'_j\in \arc{p_i, p_{i+1}}$, which is a contradiction because $p'_j \in \SuppPv' \setminus  \SuppPv$ by~\tref{Definition}{defn:G} and $\arc{p_{i-1}, p_i}\cap  \arc{p_i, p_{i+1}} = \{p_i\}$ (recall that $p_{i-1}=p_h\neq p_k=p_{i+1}$). This proves~\refp{lem:nec:5:1}. 

Note that~\refp{lem:nec:5:4} follows readily from~\refp{lem:nec:5:1}. 

By~\refp{lem:nec:5:2} and~\refp{lem:nec:5:3},  any (simple) path of $\G(\ExtPv,\ExtPv')$ connecting two nodes of $\SuppPv\setminus \SuppPv'$ is of the form $p_i$, $p'_j$, $p_{i+1}$, $p'_{j+1}$, $\dots $, $p_{i+\ell}$ for some $\ell \in \NN$, $i\in\I[r]$ and $j\in\I[r']$, where $p_{i+h}\in \arc{p'_{j+h-1}, p'_{j+h}}$ for any $h\in \I[\ell-1]$, and  $p'_{j+h-1}\in \arc{p_{i+h-1}, p_{i+h}}$ and 
\begin{equation}\label{IntervalPath}
\arc{p_{i}, p_{i+h}}\cap (\SuppPv\cup \SuppPv')=\{p_i, p'_j, p_{i+1}, p'_{j+1},\dots ,p_{i+h}\}
\end{equation}  
for any $h\in \I[\ell]$ (see~\tref{Example}{ExampleReferee}). Thus, if a component of $\G(\ExtPv,\ExtPv')$ is a cycle, we can take $p_{i+\ell}=p_i$ in~\eqref{IntervalPath}, which then implies $\I \cap (\SuppPv\cup \SuppPv')=\{p_i, p'_j, p_{i+1}, p'_{j+1},\dots ,p'_{i+l-1}\}$, that is, the set of nodes of the cycle is equal to $\SuppPv\cup \SuppPv'$. Finally, since by~\tref{Definition}{defn:G} the nodes of $\G(\ExtPv,\ExtPv')$ belong either to $\SuppPv\setminus \SuppPv'$ or to $\SuppPv' \setminus \SuppPv$, we conclude that $\SuppPv\cap \SuppPv' = \emptyset$. This shows~\refp{lem:nec:5:5}. \qedhere
\end{proof}

\begin{figure}
\begin{center}
\begin{picture}(0,0)%
\includegraphics{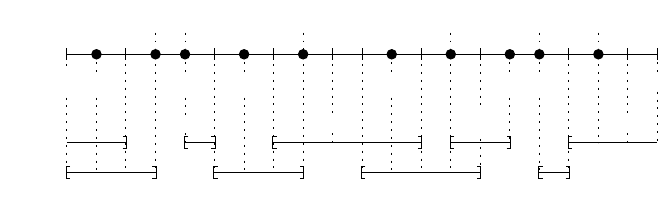}%
\end{picture}%
\setlength{\unitlength}{4144sp}%
\begingroup\makeatletter\ifx\SetFigFont\undefined%
\gdef\SetFigFont#1#2#3#4#5{%
  \reset@font\fontsize{#1}{#2pt}%
  \fontfamily{#3}\fontseries{#4}\fontshape{#5}%
  \selectfont}%
\fi\endgroup%
\begin{picture}(5022,1566)(6691,-9769)
\put(6706,-8686){\makebox(0,0)[lb]{\smash{{\SetFigFont{12}{14.4}{\rmdefault}{\mddefault}{\updefault}{\color[rgb]{0,0,0}$\I$}%
}}}}
\put(6751,-9451){\makebox(0,0)[lb]{\smash{{\SetFigFont{12}{14.4}{\rmdefault}{\mddefault}{\updefault}{\color[rgb]{0,0,0}$\cov$}%
}}}}
\put(10036,-8386){\makebox(0,0)[lb]{\smash{{\SetFigFont{12}{14.4}{\rmdefault}{\mddefault}{\updefault}{\color[rgb]{0,0,0}$p'_4$}%
}}}}
\put(11161,-8386){\makebox(0,0)[lb]{\smash{{\SetFigFont{12}{14.4}{\rmdefault}{\mddefault}{\updefault}{\color[rgb]{0,0,0}$p'_6$}%
}}}}
\put(10711,-8386){\makebox(0,0)[lb]{\smash{{\SetFigFont{12}{14.4}{\rmdefault}{\mddefault}{\updefault}{\color[rgb]{0,0,0}$p'_5$}%
}}}}
\put(8056,-8386){\makebox(0,0)[lb]{\smash{{\SetFigFont{12}{14.4}{\rmdefault}{\mddefault}{\updefault}{\color[rgb]{0,0,0}$p'_2$}%
}}}}
\put(8956,-8386){\makebox(0,0)[lb]{\smash{{\SetFigFont{12}{14.4}{\rmdefault}{\mddefault}{\updefault}{\color[rgb]{0,0,0}$p'_3$}%
}}}}
\put(7831,-8386){\makebox(0,0)[lb]{\smash{{\SetFigFont{12}{14.4}{\rmdefault}{\mddefault}{\updefault}{\color[rgb]{0,0,0}$p'_1$}%
}}}}
\put(7426,-9691){\makebox(0,0)[lb]{\smash{{\SetFigFont{12}{14.4}{\rmdefault}{\mddefault}{\updefault}{\color[rgb]{0,0,0}$C_1$}%
}}}}
\put(8551,-9691){\makebox(0,0)[lb]{\smash{{\SetFigFont{12}{14.4}{\rmdefault}{\mddefault}{\updefault}{\color[rgb]{0,0,0}$C_3$}%
}}}}
\put(9811,-9691){\makebox(0,0)[lb]{\smash{{\SetFigFont{12}{14.4}{\rmdefault}{\mddefault}{\updefault}{\color[rgb]{0,0,0}$C_5$}%
}}}}
\put(10801,-9691){\makebox(0,0)[lb]{\smash{{\SetFigFont{12}{14.4}{\rmdefault}{\mddefault}{\updefault}{\color[rgb]{0,0,0}$C_7$}%
}}}}
\put(7381,-8881){\makebox(0,0)[lb]{\smash{{\SetFigFont{12}{14.4}{\rmdefault}{\mddefault}{\updefault}{\color[rgb]{0,0,0}$p_1$}%
}}}}
\put(8056,-8881){\makebox(0,0)[lb]{\smash{{\SetFigFont{12}{14.4}{\rmdefault}{\mddefault}{\updefault}{\color[rgb]{0,0,0}$p_2$}%
}}}}
\put(8506,-8881){\makebox(0,0)[lb]{\smash{{\SetFigFont{12}{14.4}{\rmdefault}{\mddefault}{\updefault}{\color[rgb]{0,0,0}$p_3$}%
}}}}
\put(9631,-8881){\makebox(0,0)[lb]{\smash{{\SetFigFont{12}{14.4}{\rmdefault}{\mddefault}{\updefault}{\color[rgb]{0,0,0}$p_4$}%
}}}}
\put(10531,-8881){\makebox(0,0)[lb]{\smash{{\SetFigFont{12}{14.4}{\rmdefault}{\mddefault}{\updefault}{\color[rgb]{0,0,0}$p_5$}%
}}}}
\put(10756,-8881){\makebox(0,0)[lb]{\smash{{\SetFigFont{12}{14.4}{\rmdefault}{\mddefault}{\updefault}{\color[rgb]{0,0,0}$p_6$}%
}}}}
\put(11521,-8881){\makebox(0,0)[lb]{\smash{{\SetFigFont{12}{14.4}{\rmdefault}{\mddefault}{\updefault}{\color[rgb]{0,0,0}$n=21$}%
}}}}
\put(7156,-8881){\makebox(0,0)[lb]{\smash{{\SetFigFont{12}{14.4}{\rmdefault}{\mddefault}{\updefault}{\color[rgb]{0,0,0}$1$}%
}}}}
\put(8101,-9196){\makebox(0,0)[lb]{\smash{{\SetFigFont{12}{14.4}{\rmdefault}{\mddefault}{\updefault}{\color[rgb]{0,0,0}$C_2$}%
}}}}
\put(10216,-9196){\makebox(0,0)[lb]{\smash{{\SetFigFont{12}{14.4}{\rmdefault}{\mddefault}{\updefault}{\color[rgb]{0,0,0}$C_6$}%
}}}}
\put(11476,-9196){\makebox(0,0)[lb]{\smash{{\SetFigFont{12}{14.4}{\rmdefault}{\mddefault}{\updefault}{\color[rgb]{0,0,0}$C_8$}%
}}}}
\put(9226,-9196){\makebox(0,0)[lb]{\smash{{\SetFigFont{12}{14.4}{\rmdefault}{\mddefault}{\updefault}{\color[rgb]{0,0,0}$C_4$}%
}}}}
\end{picture}%
\caption{The clutter and the supports of the vertices $\ExtPv$ and $\ExtPv'$ considered in~\tref{Example}{ExampleReferee}.}\label{figure1}
\end{center}
\end{figure} 

\begin{figure}
\begin{center}
\begin{picture}(0,0)%
\includegraphics{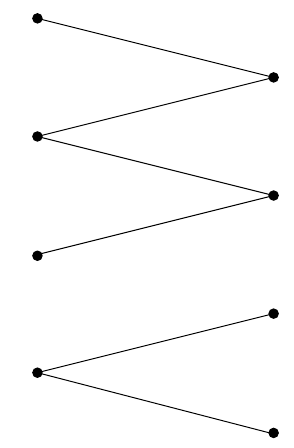}%
\end{picture}%
\setlength{\unitlength}{4144sp}%
\begingroup\makeatletter\ifx\SetFigFont\undefined%
\gdef\SetFigFont#1#2#3#4#5{%
  \reset@font\fontsize{#1}{#2pt}%
  \fontfamily{#3}\fontseries{#4}\fontshape{#5}%
  \selectfont}%
\fi\endgroup%
\begin{picture}(2235,3411)(8266,-8059)
\put(10486,-7981){\makebox(0,0)[lb]{\smash{{\SetFigFont{12}{14.4}{\rmdefault}{\mddefault}{\updefault}{\color[rgb]{0,0,0}$p'_6$}%
}}}}
\put(8281,-7531){\makebox(0,0)[lb]{\smash{{\SetFigFont{12}{14.4}{\rmdefault}{\mddefault}{\updefault}{\color[rgb]{0,0,0}$p_1$}%
}}}}
\put(8281,-6631){\makebox(0,0)[lb]{\smash{{\SetFigFont{12}{14.4}{\rmdefault}{\mddefault}{\updefault}{\color[rgb]{0,0,0}$p_3$}%
}}}}
\put(8281,-4831){\makebox(0,0)[lb]{\smash{{\SetFigFont{12}{14.4}{\rmdefault}{\mddefault}{\updefault}{\color[rgb]{0,0,0}$p_5$}%
}}}}
\put(8281,-5731){\makebox(0,0)[lb]{\smash{{\SetFigFont{12}{14.4}{\rmdefault}{\mddefault}{\updefault}{\color[rgb]{0,0,0}$p_4$}%
}}}}
\put(10486,-5281){\makebox(0,0)[lb]{\smash{{\SetFigFont{12}{14.4}{\rmdefault}{\mddefault}{\updefault}{\color[rgb]{0,0,0}$p'_4$}%
}}}}
\put(10486,-6181){\makebox(0,0)[lb]{\smash{{\SetFigFont{12}{14.4}{\rmdefault}{\mddefault}{\updefault}{\color[rgb]{0,0,0}$p'_3$}%
}}}}
\put(10486,-7081){\makebox(0,0)[lb]{\smash{{\SetFigFont{12}{14.4}{\rmdefault}{\mddefault}{\updefault}{\color[rgb]{0,0,0}$p'_1$}%
}}}}
\end{picture}%
\caption{The joint saturation graph $\G(\ExtPv,\ExtPv')$ of the vertices $\ExtPv$ and $\ExtPv'$ of~\tref{Example}{ExampleReferee}.}\label{figure2}
\end{center}
\end{figure} 

\begin{exam}\label{ExampleReferee} 
Let us consider the set covering polyhedron associated with the clutter $\cov=\{C_1,\ldots ,C_8\}$, where $C_1=\arc{1,4}$, $C_2=\arc{5,6}$, $C_3=\arc{6,9}$, $C_4=\arc{8,13}$, $C_5=\arc{11,15}$, $C_6=\arc{14,16}$, $C_7=\arc{17,18}$, $C_8=\arc{18,3}$ and $n=21$ (this clutter is represented in Figure~\ref{figure1}). The joint saturation graph $\G(\ExtPv,\ExtPv')$ of two vertices of this polyhedron is depicted in Figure~\ref{figure2} (the supports $\SuppPv=\{p_1,\ldots ,p_6\}$ and $\SuppPv'=\{p'_1,\ldots ,p'_6\}$ of these vertices are represented in Figure~\ref{figure1}). By properties~\refp{lem:nec:5:2} and~\refp{lem:nec:5:3} of \tref{Lemma}{lem:nec:5}, with each path of $\G(\ExtPv,\ExtPv')$ it is possible to associate a sequence of consecutive circular arcs of $\I$ such that in this sequence there is precisely one circular arc for each edge of the path and the endpoints of each circular arc are the nodes defining the corresponding edge of $\G(\ExtPv,\ExtPv')$ (here, we call $i$ and $j$ the {\em endpoints} of the circular arc $\arc{i, j}$ of $\I$, and we say that two circular arcs are {\em consecutive} if their intersection is one of their endpoints). Besides, each circular arc of this sequence has the property that only its endpoints belong to $\SuppPv\cup \SuppPv'$ (one of the endpoints belongs to $\SuppPv\setminus \SuppPv'$ and the other one to $\SuppPv' \setminus \SuppPv$). For the path $p_3$, $p'_3$, $p_4$, $p'_4$, $p_5$ of Figure~\ref{figure2}, the associated sequence of circular arcs is $\arc{p_3 ,p'_3}$, $\arc{p'_3, p_4}$, $\arc{p_4, p'_4}$, $\arc{p'_4, p_5}$, and for the path $p'_6$, $p_1$, $p'_1$, the associated sequence of circular arcs is $\arc{p'_6, p_1}$, $\arc{p_1, p'_1}$ (see Figure~\ref{figure1}). We refer the reader to~\tref{Example}{ExampleDifTypes} below for more examples. 
\end{exam}

\begin{lem}
\label{lem:nec:6}
Let $\ExtPv$ and $\ExtPv'$ be distinct vertices of $\Q(A)$.
Suppose $p_i$ and $p_{i+1}$ are in the same component $F$ of the joint saturation graph $\G(\ExtPv,\ExtPv')$, but do not have a common neighbor in $\SuppPv' \setminus \SuppPv$.
Then, $\SuppPv \subset F$, and so in particular $\SuppPv \cap \SuppPv' = \emptyset$.
\end{lem}

\begin{proof}
By properties~\refp{lem:nec:5:2} and~\refp{lem:nec:5:3} of \tref{Lemma}{lem:nec:5}, the path from $p_i$ to $p_{i+1}$ must contain all the elements of $\SuppPv$, and so these must be in $\SuppPv\setminus \SuppPv'$.
\end{proof}

\begin{lem}
\label{lem:nec:3}
Let $\ExtPv$ and $\ExtPv'$ be distinct vertices of $\Q(A)$. Suppose the inclusion $\arc{p'_j,p'_{j+1}}\subset \arc{p_i,p_{i+1}}$ is satisfied. Then, if $p_i \ngh p'_{j-1}$, we must have $p_i\ngh p'_j$. Similarly, if $p_{i+1}\ngh p'_{j+2}$ then $p_{i+1}\ngh p'_{j+1}$.
\end{lem}

\begin{proof}
Assume $p_i \ngh p'_{j-1}$. As $\arc{p'_j,p'_{j+1}}\subset \arc{p_i,p_{i+1}}$, observe that if $C_t\in\cov$ is such that $\{p_i,p'_{j+1}\} \subset C_t$, then $\{p_{i+1},p'_{j}\} \cap C_t\neq \emptyset$. Thus, we conclude that $p_i \nngh p'_{j+1}$, and so $p'_{j-1}\neq p'_{j+1}$.

As $p'_{j-1}\neq p'_{j+1}$, we must have $p'_{j-1}\not \in \arc{p_i,p'_j}$, because otherwise we would have $\arc{p'_{j-1},p'_{j+1}}\subset \arc{p_i,p_{i+1}}$, and so $C_t\cap \SuppPv=\emptyset$ for any $C_t\in\cov$ satisfying $C_t\cap \SuppPv' = \{p'_j\}$ (at least one of such $C_t$ exists), contradicting the fact that $\SuppPv$ is a transversal. Since $p_i$ is different from $p'_{j-1}$ and $p'_j$ (because $p_i \ngh p'_{j-1}$, and so $p_i\in \SuppPv \setminus \SuppPv'$), observe that $p'_{j-1}\not \in \arc{p_i,p'_j}$ is equivalent to $p'_j\not \in \arc{p'_{j-1},p_i}$.

As $p_i \ngh p'_{j-1}$ and $p'_j\not \in \arc{p'_{j-1},p_i}$, we have $\arc{p'_{j-1},p_i}\cap \SuppPv =\{p_i\}$ and $\arc{p'_{j-1},p_i}\cap \SuppPv' =\{p'_{j-1}\}$ (see \tref{Remark}{rem:graph:2}). Thus, we conclude that $\arc{p'_{j-1},p'_j}\cap \SuppPv =  \{p_i\}$ and $\arc{p'_j,p'_{j+1}-1}\cap \SuppPv =\emptyset$ because  $\arc{p'_j,p'_{j+1}-1}\subset \arc{p_i+1,p_{i+1}-1}$. Then, if $C_t\in\cov$ is such that $C_t\cap \SuppPv' = \{p'_j\}$, we must have $C_t\cap \SuppPv = \{p_i\}$ because $C_t$ is a circular arc and it must intersect $\SuppPv$. Therefore, recalling \tref{Definition}{defn:G}, we have $p_i\ngh p'_j$.\qedhere
\end{proof}

\begin{lem}
\label{lem:nec:7}
Let $\ExtPv$ and $\ExtPv'$ be distinct vertices of $\Q(A)$. If $\SuppPv$ is contained in a component of the joint saturation graph $\G(\ExtPv,\ExtPv')$, then $\G(\ExtPv,\ExtPv')$ is either connected or almost-connected. Moreover, in the latter case, we have $\card{\SuppPv} = \card{\SuppPv'}$.
\end{lem}

\begin{proof}
Suppose $\G(\ExtPv,\ExtPv')$ is not connected. Then, if $F$ is the component containing $\SuppPv$, by \tref{Lemma}{lem:nec:5} we conclude that $F$ must be a path connecting two consecutive elements of $\SuppPv$, say $p_i$ and $p_{i+1}$, and these elements do not have a common neighbor in $\arc{p_i, p_{i+1}}$.

Observe that $\arc{p_i, p_{i+1}}$ cannot contain more than two elements of $\SuppPv'$. Indeed, assuming the contrary we would have
$\arc{p'_{j-1},p'_{j+1}}\subset \arc{p_i, p_{i+1}}$ for some $j$ ($p'_{j-1}\neq p'_{j+1}$), and so there would exist $C_t\in\cov$ such that $C_t\cap \SuppPv=\emptyset$ (this would hold for any $C_t\in\cov$ such that $C_t\cap \SuppPv'=\{p'_j \}$), contradicting that $\SuppPv$ is a transversal.

If $\arc{p_i, p_{i+1}}$ did not contain elements of $\SuppPv'$, by \tref{Lemma}{lem:nec:5} we would conclude that $\G(\ExtPv,\ExtPv')$ is the path $F$, contradicting our assumption that
$\G(\ExtPv,\ExtPv')$ is not connected.

If there were two elements, say $p'_j$ and $p'_{j+1}$,
\tref{Lemma}{lem:nec:3} would show that $p_i\ngh p'_{j}$ and $p_{i+1}\ngh p'_{j+1}$, contradicting again that $\G(\ExtPv,\ExtPv')$ is not connected.

Thus, there can only be exactly one element of $\SuppPv'$ in $\arc{p_i, p_{i+1}}$, say $p'_j$, and this element cannot belong to $F$ since we assume that $\G(\ExtPv,\ExtPv')$ is not connected.
Therefore, $\G(\ExtPv,\ExtPv')$ consists of the isolated node $p'_j$ and the path $F$ connecting $p_i$ with $p_{i+1}$, and so it is almost-connected and $\card{\SuppPv} = \card{\SuppPv'}$.
\end{proof}

Before proving a characterization of vertex adjacency for $\Q(A)$ when $A\in\BB^{m\times n}$ is row circular, we note that with this aim we can restrict our analysis to the case where $A$ has at most three ones per row, and the vertices $\ExtPv$ and $\ExtPv'$ of $\Q(A)$ satisfy $\SuppPv \cup \SuppPv'=\I$ (in this case, observe that if $p_i \ngh p'_j$, then we must have either $p'_{j}=p_i-1$ or $p'_{j}=p_i+1$). This follows from \tref{Remark}{rem:contraction} and the fact that for a row circular matrix $A$, the contraction minor $A/I$ of \tref{Remark}{rem:contraction} has at most three ones per row, as shown in the next lemma.

\begin{lem}\label{Lemma:PropContrac}
Let $\ExtPv$ and $\ExtPv'$ be distinct vertices of $\Q(A)$, and let $I=\I\setminus (\SuppPv \cup \SuppPv')$. Then,  each row of the contraction minor $A/I$ has at most three ones.
\end{lem}

\begin{proof}
Let us denote by $\bar{C}_t$ the support of the $t$-th row of $A/I$. 

For the sake of simplicity, in this proof we will assume that the $t$-th row of $A/I$ corresponds to the $t$-th row of $A$ (recall that when contracting, dominating rows are eliminated, so this does not necessarily hold, but it can be assumed without loss of generality).

Since by \tref{Lemma}{lem:nec:1} each circular arc $C_t\in\cov$ can contain at most two elements of the support of each vertex of $\Q(A)$,
after the contraction, the resulting arc $\bar{C}_t$ can contain at most four elements, two of them corresponding to elements of $\SuppPv$, and two corresponding to elements of $\SuppPv'$.

\begin{figure}
\begin{center}
\begin{picture}(0,0)%
\includegraphics{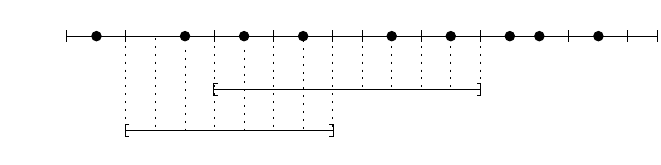}%
\end{picture}%
\setlength{\unitlength}{4144sp}%
\begingroup\makeatletter\ifx\SetFigFont\undefined%
\gdef\SetFigFont#1#2#3#4#5{%
  \reset@font\fontsize{#1}{#2pt}%
  \fontfamily{#3}\fontseries{#4}\fontshape{#5}%
  \selectfont}%
\fi\endgroup%
\begin{picture}(5022,1251)(6691,-9589)
\put(6706,-8686){\makebox(0,0)[lb]{\smash{{\SetFigFont{12}{14.4}{\rmdefault}{\mddefault}{\updefault}{\color[rgb]{0,0,0}$\I$}%
}}}}
\put(7156,-8521){\makebox(0,0)[lb]{\smash{{\SetFigFont{12}{14.4}{\rmdefault}{\mddefault}{\updefault}{\color[rgb]{0,0,0}$1$}%
}}}}
\put(8056,-8521){\makebox(0,0)[lb]{\smash{{\SetFigFont{12}{14.4}{\rmdefault}{\mddefault}{\updefault}{\color[rgb]{0,0,0}$l$}%
}}}}
\put(8506,-8521){\makebox(0,0)[lb]{\smash{{\SetFigFont{12}{14.4}{\rmdefault}{\mddefault}{\updefault}{\color[rgb]{0,0,0}$h$}%
}}}}
\put(8956,-8521){\makebox(0,0)[lb]{\smash{{\SetFigFont{12}{14.4}{\rmdefault}{\mddefault}{\updefault}{\color[rgb]{0,0,0}$i$}%
}}}}
\put(9631,-8521){\makebox(0,0)[lb]{\smash{{\SetFigFont{12}{14.4}{\rmdefault}{\mddefault}{\updefault}{\color[rgb]{0,0,0}$j$}%
}}}}
\put(10081,-8521){\makebox(0,0)[lb]{\smash{{\SetFigFont{12}{14.4}{\rmdefault}{\mddefault}{\updefault}{\color[rgb]{0,0,0}$k$}%
}}}}
\put(11656,-8521){\makebox(0,0)[lb]{\smash{{\SetFigFont{12}{14.4}{\rmdefault}{\mddefault}{\updefault}{\color[rgb]{0,0,0}$n$}%
}}}}
\put(9271,-9196){\makebox(0,0)[lb]{\smash{{\SetFigFont{12}{14.4}{\rmdefault}{\mddefault}{\updefault}{\color[rgb]{0,0,0}$C_t$}%
}}}}
\put(8326,-9511){\makebox(0,0)[lb]{\smash{{\SetFigFont{12}{14.4}{\rmdefault}{\mddefault}{\updefault}{\color[rgb]{0,0,0}$C_s$}%
}}}}
\end{picture}%
\caption{Illustration of the proof of~\tref{Lemma}{Lemma:PropContrac} (the elements of $\SuppPv \cup \SuppPv'$ are represented by dots).}\label{figure3}
\end{center}
\end{figure} 

Assume that an arc $\bar{C}_t$ with four elements exists,
and let $h$, $i$, $j$ and $k$ be the corresponding elements of $\SuppPv \cup \SuppPv'$. Without loss of generality, suppose that $\arc{i,j}\subset \arc{h,k}\subset C_t$, and that $i\in\SuppPv$ (see 
Figure~\ref{figure3} for an illustration of the relative positions of the main elements that appear in this proof). 

Let $C_s\in\cov$ be such that $C_s\cap\SuppPv = \{i\}$. Note that $(C_s \setminus C_t)\cap \SuppPv'$ cannot be empty because $C_s\cap\SuppPv = \{i\}\subsetneq C_t\cap\SuppPv$ and $A/I$ has no dominating rows. So, let $l$ be an element of $(C_s \setminus C_t)\cap \SuppPv'$. As $C_s$ is a circular arc and $\{l,i\}\subset C_s$, we either have $\{l,h,i\}\subset \arc{l,i} \subset C_s$ or $\{i,j,k,l\}\subset \arc{i,l} \subset C_s$ (recall that $\arc{i,j}\subset \arc{h,k}$ and $l\not \in C_t \supset \arc{h,k}$). Then, by the choice of $C_s$ ($i$ is the only element of $\SuppPv$ which belongs to $C_s$), we either have $\{l,h\}\subset C_s \cap \SuppPv'$ or $\{j,k,l\}\subset C_s \cap \SuppPv'$. It follows that $\{l,h\}\subset C_s \cap \SuppPv'$, because $\{j,k,l\}\subset C_s \cap \SuppPv'$ contradicts~\tref{Lemma}{lem:nec:1}. 
We conclude that $\{l,h,i\}\subset \arc{l,i} \subset C_s$ and $h\in \SuppPv'$.

Now, let $C_r\in\cov$ be such that $C_r\cap\SuppPv' = \{h\}$. Note that $l\not \in C_r$, because $l\neq h$ and $l\in \SuppPv'$ by the definition of $l$ (see the previous paragraph). Besides, since $C_r$ is a circular arc and $l\not \in C_r$, we have $k\not \in C_r$. Indeed, otherwise (i.e., if $k\in C_r$) we would have $\{h,i,j,k\} \subset C_r$ (recall that $\arc{i,j} \subset \arc{h,k}$, $h\in C_r$ and $l\not \in \arc{h,k}$), and then  by the choice of $C_r$ ($h$ is the only element of $\SuppPv'$ in $C_r$) we could conclude that $\{i,j,k\} \subset C_r \cap \SuppPv$, contradicting \tref{Lemma}{lem:nec:1}. 

Note that $\arc{l,k} \subset C_s\cup C_t$ because $h \in \arc{l,i} \subset C_s$ and $\arc{h,k}\subset C_t$. Besides, we have $\arc{l,k} \cap (\SuppPv \cup \SuppPv')=\{l,h,i,j,k\}$. Indeed, $\{l,h,i,j,k\} \subset \arc{l,k} \cap (\SuppPv \cup \SuppPv')$ due to the fact that $l\not \in C_t \supset \arc{h,k} \supset \arc{i,j}$, and if this intersection contained another element, this element should belong to $C_s$ (by our assumption we know that $C_t\cap (\SuppPv \cup \SuppPv')=\{h,i,j,k\}$), and then also to $\SuppPv'$ ($i$ is the only element of $\SuppPv$ which belongs to $C_s$), which together with the fact that $\{l,h\}\subset C_s \cap \SuppPv'$ would contradict \tref{Lemma}{lem:nec:1}. Finally,  observe that $C_r\subset \arc{l,k}$ because $C_r$ is a circular arc which contains $h$ but, by the previous paragraph, does not contain $l$ nor $k$. Thus, we conclude that $C_r\cap (\SuppPv \cup \SuppPv')\subset \{h,i,j\} \subsetneq \{h,i,j,k\} = C_t\cap (\SuppPv \cup \SuppPv')$, which contradicts the fact that $A/I$ has no dominating rows. 
\end{proof}

We need finally the next lemma to prove a characterization of vertex adjacency for $\Q(A)$.

\begin{lem}\label{lem:nec:noaristas}
Two vertices $\ExtPv$ and $\ExtPv'$ of $\Q(A)$ are not adjacent in $\Q(A)$ if their joint saturation graph $\G(\ExtPv,\ExtPv')$ has no edges.
\end{lem}

\begin{proof}
As explained above, to prove this result we may assume that $A$ has at most three ones per row and $\SuppPv \cup \SuppPv'=\I$.

By~\tref{Remark}{rem:graph:1}, $\G(\ExtPv,\ExtPv')$ has at least three nodes, so
let us fix $p_i\in \SuppPv \setminus \SuppPv'$ and $C_t\in\cov$ satisfying $C_t \cap \SuppPv = \{p_i\}$.
Then, by \tref{Lemma}{lem:nec:4} we have $\card{C_t\cap (\SuppPv' \setminus \SuppPv)} = 2$ and $C_t\cap (\SuppPv' \setminus \SuppPv) = C_t\cap \SuppPv'$. Since $\SuppPv \cup \SuppPv'=\I$, it follows that at least one of $p_i-1$ or $p_i+1$ is in $\SuppPv'\setminus\SuppPv$. Without loss of generality, assume $p'_j=p_i+1$ is in $\SuppPv' \setminus \SuppPv$.

We claim that the sets
\[
   X = (\SuppPv \setminus \{p_i\})\cup \{p'_j\}
   \quad\text{and}\quad
   X' = (\SuppPv'\setminus \{p'_j\})\cup \{p_i\}
\]
are transversals.
To see this, assume on the contrary that, for example, $X\cap C_s=\emptyset $ for some $C_s\in \cov$. Then, since $\SuppPv$ is a transversal, we must have $C_s \cap \SuppPv = \{p_i\}$. Using  the fact that $A$ has at most three ones per row and that $p_j'=p_i+1\not \in C_s$, by \tref{Lemma}{lem:nec:4} it follows that $\{p_i-2,p_i-1\} = (\SuppPv' \setminus \SuppPv)\cap C_s$. Now, taking $C_r\in  \cov$ such that $C_r\cap \SuppPv' =\{ p_i-1\}$, we necessarily have $C_r\cap \SuppPv =\{ p_i\}$ because $A$ is row circular and $\{p_i-2,p_i+1\}\subset \SuppPv' \setminus \SuppPv$.
This implies $p_i \ngh p'_{j-1}=p_i-1$,
which contradicts that $\G(\ExtPv,\ExtPv')$ has no edges. This proves that $X$ is a transversal. Similarly, it can be shown that $X'$ is also an transversal.

Observe that $X$ cannot coincide with $\SuppPv'$ since we are exchanging just one element of $\SuppPv' \setminus \SuppPv$, but $\card{\SuppPv' \setminus \SuppPv}\ge 2$ by \tref{Lemma}{lem:nec:4}.

The lemma follows now from \tref{Proposition}{propo:adys:2} defining $x=\car(X)$, $x'=\car(X')$, $d=\car(\{p_i\})$ and $d'=\car(\{p'_j\})$.\qedhere
\end{proof}

We are now ready to prove a characterization of vertex adjacency for $\Q(A)$.

\begin{thm}\label{thm:CharactAdj}
Let $A\in \BB^{m\times n}$ be a row circular matrix. Let $\ExtPv$ and $\ExtPv'$ be distinct vertices of $\Q(A)$, and $\G(\ExtPv,\ExtPv')$ be their joint saturation graph.

The vertices $\ExtPv$ and $\ExtPv'$ are adjacent in $\Q(A)$ if, and only if, one of the following conditions is satisfied:
\begin{itemize}
\item
$\G(\ExtPv,\ExtPv')$ is connected.
\item
$\G(\ExtPv,\ExtPv')$ is almost-connected and $\SuppPv\cap \SuppPv'=\emptyset$.
\end{itemize}
\end{thm}

\begin{proof}
If one of the conditions above is satisfied, then $\G(\ExtPv,\ExtPv')$ is partite-connected, and so $\ExtPv$ and $\ExtPv'$ are adjacent in $\Q(A)$ by \tref{Theorem}{thm:suf}.

Assume now that none of the conditions above is satisfied, and let us show that $\ExtPv$ and $\ExtPv'$ are not adjacent in $\Q(A)$. With this aim, as in the proof of \tref{Lemma}{lem:nec:noaristas}, we may assume that $A$ has at most three ones per row (thus, by \tref{Assumptions}{assums:A}, $A$ has between two and three ones per row) and $\SuppPv \cup \SuppPv'=\I$. 

If $\G(\ExtPv,\ExtPv')$ has no edges, the result follows from \tref{Lemma}{lem:nec:noaristas},
so we next assume that $\G(\ExtPv,\ExtPv')$ contains at least one edge.  

Let $F$ be a component containing an edge of $\G(\ExtPv,\ExtPv')$. 
Since $\G(\ExtPv,\ExtPv')$ is not connected, by \tref{Lemma}{lem:nec:5} we know that $F$ is a path, not a cycle. 

In order to show that $\ExtPv$ and $\ExtPv'$ are not adjacent we will use \tref{Proposition}{propo:adys:2}.
For doing this, we let $R = \SuppPv\cap \SuppPv'$ and define $D$ and $T$ by
\begin{subequations}
\label{equ:AdyImplicaCond:1}
\begin{equation}
D = F\cap \SuppPv , \quad 
T = \SuppPv \setminus (R \cup D).
\end{equation}
Similarly, we set
\begin{equation}
D' = F\cap \SuppPv', \quad
T' = \SuppPv' \setminus (R \cup D').
\end{equation}
Finally, we define $X$ and $X'$ by
\begin{equation}
X = R \cup T \cup D',
\quad
X' = R \cup T' \cup D.
\end{equation}
\end{subequations}

Our first aim is to prove that $X$ and $X'$ are transversals,
and we notice that it is enough to prove this only for $X$,
given the symmetry of the definitions in~\eqref{equ:AdyImplicaCond:1}.

So let us show that
\begin{equation}\label{equ:compli:-4}
C_t\cap X \ne\emptyset
\end{equation}
for any $C_t\in\cov$.

Since $\SuppPv= R \cup T \cup D$ is a transversal, it will be enough to consider just the case where $C_t$ intersects $\SuppPv$ at some element $p_i$ of $D$, i.e., assume
\begin{equation}\label{equ:compli:-3}
p_i\in C_t\cap D.
\end{equation}

If $p_i$ is connected in $\G(\ExtPv,\ExtPv')$ to two elements of $\SuppPv' \setminus \SuppPv$ (and hence of $D'$), then these are $p'_j=p_i-1$ and $p'_{j+1}=p_i+1$ because $\SuppPv \cup \SuppPv'=\I$,
and since
$C_t$ is a circular arc that contains $p_i$ and must intersect $\SuppPv'$, we have
\[
\emptyset
\neq
\{p_i-1,p_i+1\}\cap C_t 
=
\{p'_j,p'_{j+1}\}\cap C_t
\subset D' \cap C_t
\subset X \cap C_t \; ,
\]
and~\eqref{equ:compli:-4} holds.

Suppose now that $p_i$ is a leaf of the path $F$, and let
\[
p'_j\in D'
\]
be its only neighbor.

If $p'_j\in C_t$ we are done, so we next consider the case
\begin{equation}\label{equ:compli:-1}
p'_j\not \in C_t\; .
\end{equation}
Let us assume that $p'_j=p_i-1$, the case $p'_j=p_i+1$ being similar.

We claim that $p_i$ and $p_{i+1}$ cannot have a common neighbor in $\G(\ExtPv,\ExtPv')$. To see this, assume the contrary. Then, if $\card{\SuppPv} > 2$, by \trrefp{Lemma}{lem:nec:5}{lem:nec:5:3} we know that the common neighbor must belong to $\arc{p_i, p_{i+1}}$, but we have assumed that the only neighbor 
of $p_i$ is $p'_j$, and that $p'_j=p_i-1$, so it does not belong to $\arc{p_i, p_{i+1}}$. Thus, if $p_i$ and $p_{i+1}$ had a common neighbor, we must have $\card{\SuppPv} = 2$, i.e., $\SuppPv=\{p_i , p_{i+1}\}$. Note that in this case we have $\SuppPv \cap\SuppPv'=\emptyset$ (because $p_i$ and $p_{i+1}$ are nodes of $\G(\ExtPv,\ExtPv')$, and so they belong to $\SuppPv \setminus \SuppPv'$), and from \tref{Lemma}{lem:nec:7} we conclude also that $\G(\ExtPv,\ExtPv')$ is either connected or almost-connected (because $\SuppPv=\{p_i , p_{i+1}\}\subset F$). This proves our claim, since it contradicts our assumption that none of the conditions of the theorem is satisfied. 

Since $p_i$ and $p_{i+1}$ do not have a common neighbor in $\G(\ExtPv,\ExtPv')$ by the previous paragraph, observe that $p_{i+1}$ cannot belong to $F$, because otherwise by \tref{Lemma}{lem:nec:6} we could conclude that $\SuppPv \subset F$ and $\SuppPv \cap \SuppPv' = \emptyset$, and then by \tref{Lemma}{lem:nec:7} we could also conclude that $\G(\ExtPv,\ExtPv')$ is either connected or almost-connected, contradicting again our assumption that none of the conditions of the theorem is satisfied. 
Therefore, given that $p_{i+1}\in\SuppPv$, we have 
\[
p_{i+1}\in X.
\]

In order to show that
$p_{i+1}$ is also in $C_t$,
and therefore~\eqref{equ:compli:-4} holds,
let us see that the assumption
\begin{equation}\label{equ:compli:0}
p_{i+1}\notin C_t
\end{equation}
leads to a contradiction.
 
Since the $t$-th row of $A$ has between two and three ones, and we have assumed $p_i\in C_t$ in~\eqref{equ:compli:-3} and $p'_j=p_i-1\not \in C_t$ in~\eqref{equ:compli:-1}, 
it follows that either $C_t= \{ p_i, p_i+1 \}$ or $C_t=\{ p_i, p_i+1,p_i+2 \}$. What is more, as $\SuppPv \cup \SuppPv'=\I$, $p'_j=p_i-1$, $p_{i}\in \SuppPv \setminus \SuppPv'$ (since $p_i$ is a node of $\G(\ExtPv,\ExtPv')$) and we have assumed $p_{i+1}\notin C_t$ in~\eqref{equ:compli:0}, we conclude that either  
$C_t = \{ p_i, p_i+1 \}=\{ p_i, p'_{j+1} \}$ or $C_t = \{ p_i, p_i+1, p_i+2 \}=\{ p_i, p'_{j+1},p'_{j+2} \}$. Besides, note that $p'_{j+1}\in \SuppPv' \setminus \SuppPv$ in both cases, because $p'_{j+1} = p_i+1\not \in \SuppPv$ due to the fact that $p_{i+1} \not \in C_t$. Then, if $C_t=\{ p_i, p'_{j+1} \}$, we have $p_i\ngh p'_{j+1}$ by~\tref{Definition}{defn:G}. Similarly, assuming $C_t=\{ p_i, p'_{j+1},p'_{j+2} \}$, if $C_r\in\cov$ is such that $C_r\cap \SuppPv'=\{p'_{j+1}\}$,
we must have $C_r\cap \SuppPv=\{p_i\}$ (more precisely, we must have $C_r=\{p_i,p'_{j+1}\}$ because $p'_{j+1}-2=p'_j\not \in C_r$ and $p'_{j+1}+1=p'_{j+2}\not \in C_r$ by the choice of $C_r$, and the $r$-th row of $A$ has at least two ones), and so again we have $p_i\ngh p'_{j+1}$ by~\tref{Definition}{defn:G}. Thus, we can always conclude that $p_i\ngh p'_{j+1}$, which contradicts the fact that $p'_j$ is the only neighbor of $p_i$. 

Thus, the assumption~\eqref{equ:compli:0} leads to a contradiction and~\eqref{equ:compli:-4} holds, showing that $X$ and $X'$ are transversals.

Finally, we set
\[
d = \car(D), \quad d' = \car(D'), \quad
x = \ExtPv - d + d', \quad x' = \ExtPv' - d' + d,
\]
so that
$X = \supp x$
and
$X' = \supp x'$.

We notice now that
$D$ and $D'$ are not empty and different from
$\SuppPv$ and $\SuppPv'$ (respectively),
as otherwise $\SuppPv\cap \SuppPv' =\emptyset$ and $\G(\ExtPv,\ExtPv')$ would be connected or almost-connected by \tref{Lemma}{lem:nec:7},
and therefore
\[
\zero\lneqq d \lneqq v
\quad\text{and}\quad
\zero\lneqq d' \lneqq v'.
\]
Also, $x\ne\ExtPv$ since $D\neq \emptyset$ and $D\cap X=\emptyset$,
and $x\ne \ExtPv'$ since otherwise we would have $T=T'= \emptyset$ and then $\G(\ExtPv,\ExtPv')$ would be connected.

The fact that $\ExtPv$ and $\ExtPv'$ are not adjacent in $\Q(A)$ follows now from \tref{Proposition}{propo:adys:2}.
\end{proof}

\begin{exam}\label{ExampleDifTypes}
There are six types of joint saturation graphs of adjacent vertices of $\Q(A)$ when $A$ is row circular.
Among consecutive ones circulant matrices
(see \tref{Example}{exam:cnk:1}),
$\C{15}{6}$ is one of the smallest exhibiting all of these types as shown in \tref{Table}{table:cnk:2}:
disjoint supports and even path (type 1) or odd path (type 2), cycle (type 3), almost-connected (type 4), and finally overlapping supports and even path (type 5) or odd path (type 6). 

\begin{table}\centering
\begin{tabular}{*{6}{c}}
type & $\SuppPv$ & $\SuppPv'$
   & component/s \\
\hline\rule{0pt}{12pt}%
1 &
$\{1,7,13\}$    & $\{6,8,14,15\}$
   & $15, 1, 6, 7, 8, 13, 14$ (even path) \\
2 &
$\{6,12,15\}$   & $\{5,11,14\}$
   & $14, 15, 5, 6, 11, 12$  (odd path) \\
3&
$\{6,12,15\}$   & $\{5,8,14\}$
   & $5, 6, 8, 12, 14, 15, 5$ (cycle) \\
4 &
$\{6,12,15\}$   & $\{3,9,14\}$
   & $15, 3, 6, 9, 12$ (even path) + $14$ (node) \\
5 &
$\{6,12,15\}$   & $\{6,8,14,15\}$
   & $8, 12, 14$ (even path) \\
6 &
$\{6,12,15\}$   & $\{6,11,15\}$
   & $11, 12$ (odd path)
\end{tabular}
\caption{Examples showing each of the six possible behaviors of the joint saturation graph $\G(\ExtPv,\ExtPv')$ of adjacent vertices $\ExtPv$ and $\ExtPv'$ of $\Q(\C{15}{6})$.}
\label{table:cnk:2}
\end{table}
\end{exam}

\tref{Table}{table:cnk:2} also exhibits a simple consequence of our discussions:

\begin{coro}\label{coro:nec:1}
If $\ExtPv$ and $\ExtPv'$ are adjacent vertices of $\Q(A)$, then the cardinalities of their supports differ by at most one.
\end{coro}

\begin{proof}
If the joint saturation graph $\G(\ExtPv,\ExtPv')$ of $\ExtPv$ and $\ExtPv'$ is connected, then by~\tref{Lemma}{lem:nec:5} it is either a path or a cycle. Since  $\G(\ExtPv,\ExtPv')$ is bipartite with partite sets $\SuppPv \setminus \SuppPv'$ and $\SuppPv'\setminus \SuppPv$, we conclude that the cardinalities of $\SuppPv$ and $\SuppPv'$ differ by at most one. 

If $\G(\ExtPv,\ExtPv')$ is almost-connected and $\SuppPv\cap\SuppPv' = \emptyset$, either $\SuppPv$ or $\SuppPv'$ is contained in a component of $\G(\ExtPv,\ExtPv')$. Then $\card{\SuppPv} = \card{\SuppPv'}$ by~\tref{Lemma}{lem:nec:7}.
\end{proof}

The previous corollary is also a consequence of a technique by Bartholdi et al.~\cite{BOR80}, which
Eisenbrand et al.~\cite{EOSV08}
employed to show that
if $A$ is row circular, then the \emph{slices}
$\{x\in Q(A)\mid \one\cdot x = \beta\}$
are integral polytopes for $\beta\in\ZZ$.

When the matrix $A$ is not row circular, the behavior of the joint saturation graphs may be quite different, as shown by the following example, which in particular shows that being partite-connected is not a necessary condition for adjacency in the case of circulant matrices.

\begin{exam}\label{exam:ppfnd:13}
Let us consider the circulant matrix
\[
A = \cte(1,1,0,1,0,0,0,0,0,1,0,0,0) \in\RR^{13\times 13},
\]
which is the line-point incidence matrix of a non-degenerate finite projective plane of order $3$, and so it is a circulant matrix
not isomorphic to any $\cnk$. 

It turns out that if $\ExtPv$ and $\ExtPv'$ are adjacent vertices of $\Q(A)$, then their supports cannot be disjoint,
and the components of their joint saturation graph $\G(\ExtPv,\ExtPv')$
are isomorphic to a complete bipartite graph:
either $K_{1,1}$ (one edge),
or $K_{2,1}$ (path with two edges),
or $K_{3,1}$,
or $K_{3,3}$.
In particular, the nodes of $\G(\ExtPv,\ExtPv')$ may have degree more than $2$ (compare with \tref{Lemma}{lem:nec:5}).

For instance, consider the vertex $\ExtPv$ with support
$\{6, 10, 11, 13\}$, and the following choices for an adjacent vertex $\ExtPv'$:
\begin{itemize}
\item
$\ExtPv'$ with support
$\{5, 9, 10, 12\}$.
Then $\G(\ExtPv,\ExtPv')$ is isomorphic to $K_{3,3}$.

\item\label{exam:ppfnd:13:2}
$\ExtPv'$ with support
$\{4, 5, 9, 10, 11, 13\}$.
In this case, $\G(\ExtPv,\ExtPv')$ is isomorphic to $K_{3,1}$.

\item
$\ExtPv'$ with support
$\{5, 7, 8, 10, 12, 13\}$.
Then $\G(\ExtPv,\ExtPv')$ has two components, each isomorphic to $K_{2,1}$, so it is not partite-connected (and hence almost-connected), and the supports of $\ExtPv$ and $\ExtPv'$ are not disjoint.
\end{itemize}

Moreover, the supports of the vertices of $\Q(A)$ have cardinality either $4$ or $6$, so that the conclusions of \tref{Corollary}{coro:nec:1} do not hold (for instance, the previous choice of $\ExtPv$ and the second choice for $\ExtPv'$).
\end{exam}

\section{Minimally nonideal matrices}
\label{sec:mni}

A matrix $A\in\BB^{m\times n}$ is said to be \emph{ideal} if $Q(A) = \Q(A)$,
and \emph{minimally nonideal} (mni for short) if it is not ideal but $Q(A)\cap \{x \in\RR^n \mid x_i = 0\}$ and $Q(A)\cap \{x \in\RR^n \mid x_i = 1\}$ are integral polyhedra for all $i\in\I$.

There are still several interesting open questions regarding mni matrices.
On one hand, there is no good characterization of them and many studies revolve around Lehman's fundamental ideas~\cite{Le79, Le79a, Le90}.
On the other hand, few infinite families of mni matrices are known: $\cn$ for odd $n$,
the matrices corresponding to degenerate finite projective planes,
the family described by Wang~\cite{Wa11},
as well as all of the corresponding blockers of these families.

Cornu{\'e}jols and Novick~\cite{CN94} stated that, for $n$ odd and greater than $9$, it is always possible to add to $\cn$ one row so that the resulting matrix is still mni, obtaining another infinite family of mni matrices.
In this section we will apply our findings to prove this result,
showing in addition other more elaborate infinite families of mni matrices based on the family $\cn$. Let us start with the following definition.

\begin{defn}\label{defn:core}
If a binary matrix $A$ with no dominating rows and $n$ columns contains a row submatrix $A_1\in\BB^{n\times n}$ which is nonsingular and has $r$ (where $r\ge 2$) ones per row and per column, and the other rows of $A$ have more than $r$ ones, then $A_1$ is called a \emph{core} of $A$.
\end{defn}

Notice that if $A$ has a core then it is unique (up to the permutation of rows).
On the other hand, $A$ may coincide with its core.

We summarize some of Lehman's results~\cite{Le79,Le79a,Le90} on mni matrices and their consequences in the next two theorems. With this aim, let us recall that the matrix associated with the degenerate projective plane with $t+1$ points and lines is
\[
   \jt = \begin{bmatrix}
      0      & 1      & 1      & \dots  & 1 & 1 \\
      1      & 1      & 0      & \dots  & 0 & 0 \\
      1      & 0      & 1      & \dots  & 0 & 0 \\
      \vdots & \vdots & \vdots & \ddots & \vdots & \vdots \\
      1      & 0      & 0      & \cdots & 1 & 0 \\
      1      & 0      & 0      & \cdots  & 0 & 1
      \end{bmatrix}\in \BB^{(t+1)\times (t+1)}.
\]

\begin{thm}[{\cite{Le79,Le79a,Le90}}]
\label{thm:lehman:1} 
If $A\in\BB^{m\times n}$ is a mni matrix, then
$Q(A)$ has a unique fractional vertex
and the blocker of $A$, 
$\blk(A)$, is mni.
\end{thm}

\begin{thm}[{\cite{Le79,Le79a,Le90}}]
\label{thm:lehman:2}
Let $A\in\BB^{m\times n}$ be a mni matrix which is not isomorphic to $\jt$ for any $t\ge 2$.
Then $A$ has a core, say $A_1$, and its blocker 
$\blk(A)$ has a core, say $B_1$, 
such that:
\begin{enumcona}
\item
$A_1\one = r\,\one$ and $B_1\one = s\,\one$.

\item
The rows of $A_1$ and $B_1$ may be permuted so that
\begin{equation}
   \label{equ:lehman}
    A_1 B_1^\text{\upshape\textsf{T}} = \jota + (r s - n)\,\matid,
\end{equation}
where $\jota$ is the matrix of all ones and
$\matid$ is the identity matrix.

\item
$\fast = \frac{1}{r}\,\one$ is a fractional vertex of $Q(A)$. 

\item\label{thm:lehman:2:d}
$\fast$ is in exactly $n$ edges of $Q(A)$. More precisely, $\fast$ is adjacent in $Q(A)$ to exactly $n$ vertices which make up the rows of $B_1$. 

\item\label{thm:lehman:2:s}
$x\cdot\one \ge s$ defines a facet of $\Q(A)$, and $Q(A) \cap \{x \in \RR^n \mid x\cdot\one \ge s\} = \Q(A)$.

\end{enumcona}
\end{thm}
 
L{\"u}tolf and Margot~\cite{LM98} gave a condition which ensures that a binary matrix is mni. 

\begin{lem}[{\cite[Lemma~2.8]{LM98}}]\label{lem:LM98}
Suppose that $A\in\BB^{m\times n} $ has core $A_1$, with $r$ ones per row,
that its blocker $\blk(A)$ has core $B_1$, with $s$ ones per row, 
and that~\eqref{equ:lehman} holds. Then, if $Q(A)$ has just one fractional vertex, $A$ must be mni. 
\end{lem}

Despite the ``minimal'' in mni, a mni matrix may have a row submatrix which is also mni.
Cornu{\'e}jols and Novick~\cite{CN94}, and later L{\"u}tolf and Margot~\cite{LM98},
used this fact to construct many new mni matrices by adding rows to known ones.
Of interest to us here is the possibility of adding one or more rows to $\cn$, which is a mni matrix for $n$ odd, to obtain another mni matrix.

One of the main tools for studying the vertices of the polyhedron which results from the addition of an inequality to the system of inequalities describing a given polyhedron is the following variant of Lemma~8 of~\cite{FP96}, which essentially says that the new vertices are obtained by intersecting the edges of the original polyhedron with the hyperplane associated with the new inequality. 

\begin{propo}[{variant of \cite[Lemma~8]{FP96}}]
\label{propo:FK96}
Let $A \in\RR^{m\times n}$ be a matrix with non-negative entries, and suppose
$P = \{x\in\RR^n \mid Ax\ge b, x\ge\zero\}$ is a full dimensional polyhedron.
Let us further assume that the inequality $a\cdot x\ge c$ is independent of those defining $P$, where $a\ge\zero$ and $c > 0$.

Then, any vertex $\ExtPv$ of the polyhedron
$P' = P\cap \{x\in\RR^n \mid a\cdot x\ge c\}$ must satisfy one (and only one) of the following:

\begin{itemize}
\item
$\ExtPv$ is a vertex of $P$ satisfying $a\cdot \ExtPv\ge c$,

\item
$\ExtPv$ is a convex combination
$\ExtPv = \alpha \ExtPw + (1-\alpha ) \ExtPw'$ of adjacent vertices $\ExtPw$ and $\ExtPw'$ of $P$, satisfying $a\cdot \ExtPw > c$, $a\cdot \ExtPw' < c$, and $a\cdot \ExtPv = c$, that is,
$\alpha = (c - a\cdot \ExtPw') / (a\cdot \ExtPw - a\cdot \ExtPw')$,

\item
$\ExtPv = \ExtPw + \beta \ee_h$ for some
vertex $\ExtPw $ of $P$, $\beta > 0$ and $h\in\I$,
such that
$\{\ExtPw + \gamma \ee_h \mid \gamma \ge 0\}$ is an (infinite) edge of $P$,
$a\cdot \ExtPw < c$ and $a\cdot \ExtPv = c$,  that is,
$\beta = (c - a\cdot \ExtPw )/a_h$ (necessarily $a_h\ne 0$).

\end{itemize}
\end{propo}

Suppose the $mni$ matrix $A$ has core $A_1$ and its blocker $B = \blk(A)$ has core $B_1$, so that the properties of \tref{Theorem}{thm:lehman:2} are satisfied, in particular~\eqref{equ:lehman}.
Let $\mc{B}$ be the set consisting of the fractional vertex $\fast$ and the binary vertices of $Q(A)$ which are adjacent to it (i.e., the rows of $B_1$, see \tref{Theorem}{thm:lehman:2}). Suppose furthermore that the binary matrix $M$ has more than $r$ ones per row, and we add to $A$ the rows of $M$ obtaining the matrix $E$, which has no dominating rows. Schematically,
\begin{equation}
    \label{equ:append}
    E = \begin{bmatrix} A \\ M \end{bmatrix}.
\end{equation}
Then, we have:

\begin{lem}\label{lem:mni:2}
If $M\ExtPu \ge\one$ for all $\ExtPu \in\mc{B}$,
and any vertex of $Q(E)$
which is not in $\mc{B}$ is binary and has more than $s$ ones,
then $E$ is mni.
\end{lem}

\begin{proof}
Since $Q(E)\subset Q(A)$, if $\ExtPv\in Q(E)$ is a vertex of $Q(A)$, then it is also a vertex of $Q(E)$.
Thus, the elements of $\mc{B}$ are vertices of $Q(E)$ because $M\ExtPu \ge \one$ for $\ExtPu \in \mc{B}$.
Since any vertex of $Q(E)$ which is not in $\mc{B}$ is binary, we conclude that $Q(E)$ has just one fractional vertex.

By \tref{Lemma}{lem:LM98}, it is enough to show now that $E$ has core $A_1$ and $\blk(E)$ has core $B_1$.
The first condition is clear ($M$  has more than $r$ ones per row), and the second one follows from the fact that the rows of $\blk(E)$ are exactly the binary vertices of $Q(E)$, and that any vertex of $Q(E)$ which is not in $\mc{B}$ has more than $s$ ones.\qedhere
\end{proof}

The following result relates vertex adjacency in $Q(A)$ with vertex adjacency in $\Q(A)$ when $A$ is mni.

\begin{lem}\label{lem:mni:fast}
Let $A$ be a mni matrix not isomorphic to any $\jt$ ($t\ge 2$). Suppose the core $A_1$ of $A$ has $r$ ones per row, and the core $B_1$ of its blocker has $s$ ones per row.
Let $\ExtPv$ and $\ExtPv'$ be binary vertices of $Q(A)$.
Then, we have: 
\begin{enumcona}

\item\label{lem:mni:fast:a}
If $\max\,\{\ExtPv\cdot\one,  \ExtPv'\cdot\one\} > s$, the vertices $\ExtPv$ and $\ExtPv'$ are adjacent in $Q(A)$ if and only if they are adjacent in $\Q(A)$.

\item\label{lem:mni:fast:b}
If $\ExtPv\cdot\one = \ExtPv'\cdot\one = s$, the vertices $\ExtPv$ and $\ExtPv'$ are always adjacent in $\Q(A)$,
and they are adjacent
in $Q(A)$ if and only if $\supp \ExtPv\cup \supp \ExtPv' \neq \I$.
\end{enumcona}
\end{lem}

\begin{proof}
By~\trrefp{Theorem}{thm:lehman:2}{thm:lehman:2:s}, we know that $ Q(A) \cap \{x \in \RR^n \mid x\cdot\one \ge s\} = \Q(A)$. If $\max\,\{\ExtPv\cdot\one, \ExtPv'\cdot\one\} > s$, at least one of the vertices $\ExtPv$ and $\ExtPv'$ does not satisfy the inequality $x\cdot\one \ge s$ tightly. Therefore, when we add this inequality to the system $A x\ge \one$, the adjacency relation between these vertices does not change. This shows \refp{lem:mni:fast:a}.  

For the first part of \refp{lem:mni:fast:b},
we notice that $\ExtPv$ satisfies with equality $n-1$ of the inequalities corresponding to the rows of $A_1$, as $v$ is adjacent to $\fast = \frac{1}{r}\,\one$ by \trrefp{Theorem}{thm:lehman:2}{thm:lehman:2:d}.
Since this is also true for $\ExtPv'$, $\ExtPv$ and $\ExtPv'$ satisfy tightly $n-2$ inequalities coming from $A_1$ and the equality $x\cdot\one = s$ which defines a facet of $\Q(A)$ and is linearly independent with those of $A_1$ (as $\fast$ does not satisfy it). Thus, $\ExtPv$ and $\ExtPv'$ are adjacent in $\Q(A)$. 

For the last part of~\refp{lem:mni:fast:b},
assume first that $\supp \ExtPv\cup \supp \ExtPv' \neq \I$. Let $y = \sum_{k\in \I[\ell]} \lambda_k \ExtPu^k$ be a strict convex combination of vertices of $Q(A)$, and suppose $y \leq \frac{1}{2}\,(\ExtPv + \ExtPv')$. Observe that $y_h = 0$ for any $h\in \supp \ExtPv\cup \supp \ExtPv'$. Then for any of such $h$ and any $k\in \I[\ell]$, we have $\ExtPu^k_h =0$, and so $\ExtPu^k\neq \fast$, because $y$ is a strict convex combination of $\ExtPu^1,\ldots ,\ExtPu^\ell$. Therefore $\ExtPu^k$ is a binary vertex of $Q(A)$, and so of $\Q(A)$, for all $k\in \I[\ell]$. Since $\ExtPv$ and $\ExtPv'$ are adjacent in $\Q(A)$ by the previous paragraph, from the equivalence of~\refp{propo:adys:1:a} and~\refp{propo:adys:1:c} of \tref{Proposition}{propo:adys:1} when $P=\Q(A)$, it follows that $\ell=2$ and, without loss of generality, $\ExtPu^1=\ExtPv$ and $\ExtPu^2=\ExtPv'$. Using again the equivalence of~\refp{propo:adys:1:a} and~\refp{propo:adys:1:c} of \tref{Proposition}{propo:adys:1} but in this case when $P=Q(A)$, we conclude that $\ExtPv$ and $\ExtPv'$ are adjacent in $Q(A)$.

Finally, if we assume that $\supp \ExtPv\cup \supp \ExtPv' = \I$, we have $\frac{1}{2}\,(\ExtPv + \ExtPv')\ge \fast$ (since $r\ge 2$, see \tref{Definition}{defn:core}), and then by the equivalence of~\refp{propo:adys:1:a} and~\refp{propo:adys:1:c} of  \tref{Proposition}{propo:adys:1} when $P=Q(A)$, we conclude that $\ExtPv$ and $\ExtPv'$ are not adjacent in $Q(A)$.
\end{proof}

In the remainder of this section we will focus our attention on the (mni) matrix $\cn$ for $n$ odd. This matrix coincides with its core, having exactly $2$ ones per row and per column, and the core of $\blk(\cn)$ has $s = (n + 1)/2$ ones per row and per column.

It is convenient to observe that $\cn$ is the edge-node incidence matrix of the cycle graph $\circuit_n$ with $n$ nodes, and that for any pair of binary vertices $\ExtPv$ and $\ExtPv'$ of $\qcn$, the joint saturation graph $\G(\ExtPv,\ExtPv')$ coincides with the subgraph of $\circuit_n$ induced by $\SuppPv\symdif \SuppPv'$. 
 
We now present two results providing properties of the vertices of $\qcn$.

\begin{coro}\label{CoroTrivial}
Let $\ExtPv$ and $\ExtPv'$ be two binary vertices of $\qcn$. If $\ExtPv$ and $\ExtPv'$ are adjacent in $\qcn$, then their joint saturation graph $\G(\ExtPv,\ExtPv')$ is connected. Moreover, if $\max\,\{\ExtPv\cdot\one,  \ExtPv'\cdot\one\} > s=(n + 1)/2 $, then also the reverse implication holds.
\end{coro}

\begin{proof}
To prove the first part of the corollary, observe that for any $A\in\BB^{m\times n}$, two binary vertices of $Q(A)$ are adjacent in $Q(A)$ only if they are adjacent in $\Q(A)$. Then, as a consequence of the ``only if'' part of \tref{Corollary}{coro:chacarct:two:ones} applied to $G=\circuit_n$, $\ExtPv$ and $\ExtPv'$ are adjacent in $\qcn$ only if $\G(\ExtPv,\ExtPv')$ is connected (recall that the subgraph of $\circuit_n$ induced by $\SuppPv\symdif \SuppPv'$ coincides with $\G(\ExtPv,\ExtPv')$).

The second part of the corollary follows readily from \trrefp{Lemma}{lem:mni:fast}{lem:mni:fast:a} and the ``if'' part of \tref{Corollary}{coro:chacarct:two:ones} applied to $G=\circuit_n$ (or \tref{Theorem}{thm:CharactAdj}).
\end{proof}

\begin{lem}\label{lem:mni:extremo}
The only fractional vertex of $\qcn$ is $\fast =\frac{1}{2}\,\one $. Any binary vertex of $\qcn$ is the characteristic vector of a minimal node cover of $\circuit_n$, thus it has at least $s = (n + 1)/2$ ones. A point $\ExtPv\in\BB^n$ is a vertex of $\qcn$ if and only if it has neither three consecutive ones, nor two consecutive zeros. 
\end{lem}

\begin{proof}
We omit the proof, since the only non-trivial part of the lemma follows from~\tref{Theorem}{thm:lehman:1}. 
\end{proof}

In \cite{CN94}, Cornu{\'e}jols and Novick remark that $\cn$ can be extended to another mni matrix by appending to it one row of a specific form. Next, we prove this fact. 

\begin{propo}[\cite{CN94}]\label{propo:CN94}
For $n\ge 9$ odd, let $\{i,j,l\}\subset \I$
be such that
$i < j < l$,
$j - i\ge 3$ odd,
$l - j\ge 3$ odd,
and either $i\neq 1$ or $l\neq n$.
Let $a\in \BB^n$ be the characteristic vector of the set $\{i,j,l\}$.
Then, the matrix $E$ obtained by adding to $\cn$ the row vector $a$ is mni.
\end{propo}

\begin{proof}
By \tref{Lemma}{lem:mni:2} it will be enough to show that:
\begin{enumcona}

\item\label{propo:CN94:1}
If $\mc{B}$ is the set consisting of the fractional vertex $\fast =\frac{1}{2}\,\one $ and the vertices of $\qcn$ which are adjacent to it, then $a\cdot \ExtPu \ge 1$ for all $\ExtPu \in\mc{B}$.

\item\label{propo:CN94:2}
Any vertex of $Q(E)$ not in $\mc{B}$
is binary and has more than $s = (n + 1)/2$ ones.

\end{enumcona}

To show~\refp{propo:CN94:1}, notice that $a\cdot\fast = 3/2 > 1$. Moreover, if $\ExtPu$ is adjacent to $\fast$ in $\qcn$, then $\ExtPu$ is in the core of $\blk(\cn)$ by \trrefp{Theorem}{thm:lehman:2}{thm:lehman:2:d}. Therefore, we have $a\cdot \ExtPu \ge 1$ due to the fact that $\blk(\cn) = \cte (1,1,0,1,0,1,\dots,0,1,0,1,0)$ and that $\card{\arc{i,j}}$, $\card{\arc{j,l}}$ and $\card{\arc{l, i}}$ are even (the latter follows from the fact that $j - i$, $l - j$ and $n$ are odd). 

To show~\refp{propo:CN94:2} we rely on~\tref{Proposition}{propo:FK96}.

Suppose $\ExtPv$ is a vertex of $Q(E)$ which is a convex combination of the adjacent vertices $\ExtPw$ and $\ExtPw'$ of $\qcn$,
\[ 
\ExtPv = \alpha \ExtPw + (1 - \alpha ) \ExtPw', 
\]
with
\begin{equation}
   \label{equ:propo:CN94}
    a\cdot\ExtPv = 1, \quad
    a\cdot\ExtPw > 1, \quad
    a\cdot\ExtPw' < 1.
\end{equation}

Since $a\cdot\fast = 3/2$ and $a\cdot \ExtPu \ge 1$ for any vertex $\ExtPu$ of $\qcn$ which is adjacent to $\fast$, we conclude that $\fast$ is different from both $\ExtPw$ and $\ExtPw'$. Thus, $\ExtPw$ and $\ExtPw'$ are binary (because $\fast$ in the only fractional vertex of $\qcn$ by~\tref{Theorem}{thm:lehman:1}), and so  we must have
\[
a\cdot\ExtPw' = 0 \; .
\]

Now, since $\ExtPw$ and $\ExtPw'$ are adjacent in $\qcn$, by \tref{Corollary}{CoroTrivial} we know that their joint saturation graph $\G(\ExtPw, \ExtPw')$, which coincides with the subgraph of $\circuit_n$ induced by $\SuppPw\symdif \SuppPw'$, is connected. As $\G(\ExtPw, \ExtPw')$ is bipartite and $n$ is odd, $\G(\ExtPw, \ExtPw')$ cannot be equal to $\circuit_n$, and so it is a path. Suppose that $\card{(\SuppPw\symdif \SuppPw')\cap \{i,j,l\}}\geq 2$. Then $\card{(\SuppPw \setminus \SuppPw') \cap \{i,j,l\}}\geq 1$ and $\card{(\SuppPw'\setminus \SuppPw) \cap \{i,j,l\}}\geq 1$, because $\card{\arc{i,j}}$, $\card{\arc{j,l}}$ and $\card{\arc{l, i}}$ are even, and by \tref{Definition}{defn:G} each edge of the path $\G(\ExtPw, \ExtPw')$ connects a node in $\SuppPw \setminus \SuppPw'$ with a node in $\SuppPw'\setminus \SuppPw$. Thus, it is not possible to have $\card{\SuppPw'\cap \{i,j,l\}}= 0$ and $\card{\SuppPw \cap \{i,j,l\}}\geq 2$. Therefore, since $a\cdot\ExtPw' = 0$, we must have
\[
a\cdot\ExtPw \le 1,
\]
which contradicts~\eqref{equ:propo:CN94}. We conclude that the second possibility described in \tref{Proposition}{propo:FK96} cannot happen for $Q(E) = \qcn \cap \{x\in\RR^n \mid a\cdot x\ge 1\}$. 

Suppose now $\ExtPv$ is a vertex of $Q(E)$ of the form
\[
\ExtPv = \ExtPw + \beta \ee_h,
\]
where $\ExtPw$ is a vertex of $\qcn$ satisfying $a\cdot\ExtPw < 1$, and $a\cdot\ExtPv = 1$. Once again, $\ExtPw$ cannot be either $\fast$ or any of the vertices of $\qcn$ which are adjacent to it. Thus, by~\tref{Theorems}{thm:lehman:1} and~\ref{thm:lehman:2} it follows that $\ExtPw$ is binary and $s < \ExtPw\cdot\one$. Then, we have $a\cdot\ExtPw =0$, which implies $\beta = 1$ since $a\cdot\ExtPv = 1$. Thus, we conclude that $\ExtPv$ is binary, and
\[ 
s < \ExtPw\cdot\one \le \ExtPv\cdot\one \; , 
\]
proving~\refp{propo:CN94:2}. 
\end{proof}

One would hope that it is possible to add a circulant matrix $M$ instead of just a single row, but this is not true in general. For instance, if we add to $\C{15}{2}$ the matrix
\[\setcounter{MaxMatrixCols}{15}
   M = \begin{bmatrix}
   1 & 0 & 0 & 0 & 0 & 1 & 0 & 0 & 0 & 0 & 1 & 0 & 0 & 0 & 0 \\
   0 & 1 & 0 & 0 & 0 & 0 & 1 & 0 & 0 & 0 & 0 & 1 & 0 & 0 & 0 \\
   0 & 0 & 1 & 0 & 0 & 0 & 0 & 1 & 0 & 0 & 0 & 0 & 1 & 0 & 0 \\
   0 & 0 & 0 & 1 & 0 & 0 & 0 & 0 & 1 & 0 & 0 & 0 & 0 & 1 & 0 \\
   0 & 0 & 0 & 0 & 1 & 0 & 0 & 0 & 0 & 1 & 0 & 0 & 0 & 0 & 1
   \end{bmatrix},
\]
the resulting matrix $E$ is not mni, whereas, by \tref{Proposition}{propo:CN94}, adding just one row of $M$ we obtain a mni matrix.

Let us see how to obtain systematically a mni matrix by adding to $\cn$ all the rows of a circulant matrix. With this aim, let us assume $n = 3\nu$, where $\nu \in \NN$ is greater than or equal to 3. Henceforth, for each $i\in \I[3]$, we denote by $a^i$ the unique row vector of $\BB^n$ for which 
\begin{equation}\label{equ:3nu:v}
\SuppVa^i =\left\{ i, i + 3,\ldots , i + 3(\nu -1) \right\}
\end{equation}
and by $\ExtPw^i$ the unique point of $\BB^n$ for which 
\begin{equation}\label{equ:3nu:w}
\SuppPw^i = \I \setminus \SuppVa^i \; .
\end{equation}
Thus, we have for example $\ExtPw^1 = (0,1,1,0,1,1,\dots,0,1,1)$. We observe that $\ExtPw^i\cdot a^i = 0$ and (by \tref{Lemma}{lem:mni:extremo}) $\ExtPw^i$ is a vertex of $\qcn$, for each $i\in \I[3]$. Finally, we let $\mc{W} = \{\ExtPw^1,\ExtPw^2,\ExtPw^3\}$. 

We next show:

\begin{thm}\label{thm:3nu}
Let $n = 3\nu$ be odd, where $\nu\geq 3$. If $E$ is the matrix obtained from $\cn$ by appending to it the rows $a^1$, $a^2$ and $a^3$ (defined by~\eqref{equ:3nu:v}), then $E$ is mni. 
\end{thm}

The proof of this result is based on the following two lemmas, in which we preserve the notation. 

\begin{lem}\label{lem:3nu:2} 
For each $i\in \I[3]$ the following statements are equivalent:

\begin{enumcona}

\item\label{lem:3nu:2:a}
$\ExtPu$ is adjacent to $\ExtPw^i$ (defined by~\eqref{equ:3nu:w}) in $\qcn$.

\item\label{lem:3nu:2:b} 
$\ExtPu$ is binary and there exists $j \in \I$ such that $\SuppPu \setminus \SuppPw^i = \{j\}$ and $\SuppPw^i \setminus \SuppPu = \{j-1,j+1\}$ (here, as usual, the operations on the indices should be understood modulo $n$).
\end{enumcona}
\end{lem}  

\begin{proof}
In the first place, assume that~\refp{lem:3nu:2:b} holds. Then, $\ExtPu$ can be obtained by replacing in $\ExtPw^i$ a subvector of the form $(1,1,0,1,1)$ by one of the form $(1,0,1,0,1)$, and so by \tref{Lemma}{lem:mni:extremo} it follows that $\ExtPu$ is a vertex of $\qcn$. Besides, since $\G(\ExtPw^i,\ExtPu)$ consists of the path $j-1$, $j$, $j+1$ and $\ExtPw^i\cdot\one =  2\nu$ is strictly greater than the number $s = (3\nu + 1)/2$ of ones per row in the core of $\blk(\cn)$, by \tref{Corollary}{CoroTrivial}
we conclude that $\ExtPu$ and $\ExtPw^i$ are adjacent in $\qcn$. 
 
Assume now that~\refp{lem:3nu:2:a} hold. Since $\ExtPw^i\cdot\one =  2\nu > s = (3\nu + 1)/2$, by \trrefp{Theorem}{thm:lehman:2}{thm:lehman:2:d} we conclude that $\ExtPu$ is not equal to the unique fractional vertex $\fast=\frac{1}{2}\,\one$ of $\qcn$, and so $\ExtPu$ is binary. Then, as $\ExtPu$ and $\ExtPw^i$ are adjacent in $\qcn$, by \tref{Corollary}{CoroTrivial} we know that $\G(\ExtPw^i,\ExtPu)$ is connected. Let $j$ be a node of $\G(\ExtPw^i,\ExtPu)$ contained in $\SuppPu \setminus \SuppPw^i$ (such a node exists by \tref{Remark}{rem:graph:1}). Then, by the definition of $\ExtPw^i$, we necessarily have $\{j-2,j-1,j+1,j+2\} \subset \SuppPw^i$. Since $\G(\ExtPw^i, \ExtPu)$ is a connected bipartite subgraph of $\circuit_n$ with partite sets $\SuppPu \setminus \SuppPw^i$ and $\SuppPw^i \setminus \SuppPu$, it follows that besides $j$, only $j-1$ and $j+1$ can be nodes of $\G(\ExtPw^i,\ExtPu)$. Assume that $j-1$ is not a node of $\G(\ExtPw^i,\ExtPu)$, or equivalently that $j-1 \in \SuppPu$. Since $j-2$ is not a node of $\G(\ExtPw^i,\ExtPu)$, we necessarily have $j-2 \in \SuppPu$, and so $\ExtPu$ would have three consecutive ones, contradicting \tref{Lemma}{lem:mni:extremo}. It follows that $j-1 \not \in \SuppPu$. Similarly, it can be shown that $j+1 \not \in \SuppPu$. We conclude that $\SuppPu \setminus \SuppPw^i = \{j\}$ and $\SuppPw^i \setminus \SuppPu = \{j-1,j+1\}$, which completes the proof.
\end{proof}

\begin{lem}\label{LemmaVerticesOfE}
If $E$ is the matrix $\cn$ to which we have appended the rows $a^1$, $a^2$ and $a^3$ (defined by~\eqref{equ:3nu:v}), then the vertices of $Q(E)$ are those of $\qcn$ except for $\ExtPw^1$, $\ExtPw^2$ and $\ExtPw^3$ (defined by~\eqref{equ:3nu:w}).
\end{lem}

\begin{proof}
In the first place, we claim that $\ExtPv\cdot a^i \ge 1$ for every vertex $\ExtPv$ of $\qcn$ different from $\ExtPw^i$. To see this, assume $\ExtPv$ is a vertex of $\qcn$ such that $\ExtPv\cdot a^i < 1$. Then $\ExtPv$ is binary (for the unique fractional vertex $\fast=\frac{1}{2}\,\one$ of $\qcn$ we have $\fast\cdot a^i = \frac{\nu}{2} \ge \frac{3}{2}$) and $\ExtPv\cdot a^i = 0$, and so $\ExtPv$ is dominated by $\ExtPw^i$. Thus, $\ExtPv$ and $\ExtPw^i$ must coincide as they are both vertices of $\qcn$. This proves our claim. 

By our claim above we conclude that every vertex of $\qcn$ not in $\mc{W} = \{\ExtPw^1,\ExtPw^2,\ExtPw^3\}$ is a vertex of $Q(E)$.

Conversely, let us see that if we add one row $a^i$ at a time then no new vertices are created, and the points in $\mc{W}$ are the only vertices that are eliminated.

In the first place, let us consider the intersection of $\qcn$ with the half-space $\{x\in\RR^n \mid a^i\cdot x \ge 1 \}$. If new vertices are created, then by \tref{Proposition}{propo:FK96} they should come from the intersection of the hyperplane $\{x\in\RR^n \mid a^i\cdot x = 1\}$ with an edge of $\qcn$. This edge should be incident to a vertex $\ExtPv$ satisfying $a^i\cdot \ExtPv < 1$, and therefore, by our claim above, this edge is incident to $\ExtPw^i$. Given that $\ExtPw^i$ is adjacent only to the vertices which satisfy the condition of \trrefp{Lemma}{lem:3nu:2}{lem:3nu:2:b}, and that any of such vertices belongs to $\{x\in\RR^n \mid a^i\cdot x = 1\}$, from \tref{Proposition}{propo:FK96} we conclude that any new vertex must come from the intersection of an infinite edge of the form $\{\ExtPw^i + \gamma \ee_j \mid \gamma \geq 0\}$ with the hyperplane $\{x\in\RR^n \mid a^i\cdot x = 1\}$. Since $a^i\cdot\ExtPw^i = 0$, this intersection must be of the form $\{\ExtPw^i + \ee_j\}$ with $j \in \SuppVa^i =\I \setminus \SuppPw^i$. Let $\ExtPu^j$ be the vertex of $\qcn$ defined by \trrefp{Lemma}{lem:3nu:2}{lem:3nu:2:b}, i.e., let $\ExtPu^j \in \BB^n$ be such that $\SuppPu^j \setminus \SuppPw^i = \{j\}$ and $\SuppPw^i \setminus \SuppPu^j = \{j-1,j+1\}$. We observe that by the definition of $\ExtPu^j$ we have:
\[
\ExtPw^i + \ee_j = \ExtPu^j + \ee_{j-1} + \ee_{j+1}.
\]
It follows that $\ExtPw^i + \ee_j$ dominates $\ExtPu^j$, and then it cannot be a vertex of $\qcn \cap \{x\in\RR^n \mid a^i\cdot x \ge 1 \}$. Thus, we conclude that no new vertex is created and the vertices of $\qcn \cap \{x\in\RR^n \mid a^i\cdot x \ge 1 \}$ are the vertices of $\qcn$ except for $\ExtPw^i$.

Finally, since $a^i\cdot\ExtPw^h = \nu >1$ for $h\neq i$, observe that the addition of the inequality $a^i\cdot x \ge 1$ does not modify the adjacency relations for $\ExtPw^h$ ($h\neq i$) in the resulting polyhedron,  which are still characterized by \trrefp{Lemma}{lem:3nu:2}{lem:3nu:2:b}. Then, we can repeat the argument above each time we add a new inequality. This completes the proof. 
\end{proof}

\begin{proof}[Proof of \tref{Theorem}{thm:3nu}]
The previous lemmas show that, except for the points in $\mc{W}$, the vertices of $\qcn$ and $Q(E)$ coincide. Moreover, no vertex of $\mc{B}$ belongs to $\mc{W}$ (the former are either the fractional vertex $\fast$ or have $s = (3\nu + 1)/2$ ones, while the latter have $2\nu$ ones). Thus, \tref{Lemma}{lem:mni:2} yields that $E$ is mni. 
\end{proof}

\section*{Acknowledgements}
\label{sec:acknowledge}

\begin{itemize}

\item
The authors are very grateful to the anonymous reviewers for their comments and suggestions which helped to improve the presentation of the results in this paper.

\item
We made wide use of the freely available polyhedral computational codes \emph{PORTA} by Christof and L\"{o}bel~\cite{porta} and \emph{cdd} by Fukuda~\cite{Fu0.94}: our thanks to their authors.

\item
This work was partially supported by grant PIP 112-201101-01026 from Consejo Nacional de Investigaciones Cient\'ificas y T\'ecnicas (CONICET), Argentina. P.~B.~Tolomei was also partially supported by grant PID-ING 416 from Universidad Nacional de Rosario (UNR), Argentina.

\end{itemize}


\end{document}